\RequirePackage{fix-cm}
\documentclass[smallextended]{svjour3}       
\smartqed  
\usepackage{graphicx}
\usepackage{url}
\usepackage{hyperref}
\usepackage{amssymb,amsmath}
\usepackage{caption}
\usepackage{subcaption}
\DeclareMathOperator{\Iso}{Iso}
\DeclareMathOperator{\PSL}{PSL}
\DeclareMathOperator{\PGL}{PGL}
\DeclareMathOperator{\Hom}{Hom}
\DeclareMathOperator{\Aut}{Aut}
\DeclareMathOperator{\Inn}{Inn}
\DeclareMathOperator{\Out}{Out}
\DeclareMathOperator{\Mod}{Mod}

\DeclareMathOperator{\id}{id}

\begin{document}

\title{Isotopic tiling theory for hyperbolic surfaces\thanks{This research was funded by the Emmy Noether Programme of the Deutsche Forschungsgemeinschaft. B.K was supported by the Deutscher Akademischer Austauschdienst for a research stay at the Australian National University.}
}


\author{Benedikt Kolbe         \and
        Myfanwy E. Evans 
}


\institute{Benedikt Kolbe \at
               Technische Universit\"{a}t Berlin\\
               Stra\ss e des 17. Juni 136\\
				10623 Berlin\\
				\email{benedikt.kolbe@inria.fr}
           \and
           Myfanwy E. Evans \at
              Universit\"{a}t Potsdam\\
                Karl-Liebknecht-Str. 24-25\\
		14476 Potsdam-Golm\\
	\email{evans@uni-potsdam.de}
}

\maketitle

\begin{abstract}
In this paper, we develop the mathematical tools needed to explore isotopy classes of tilings on hyperbolic surfaces of finite genus, possibly nonorientable, with boundary, and punctured. More specifically, we generalize results on Delaney-Dress combinatorial tiling theory using an extension of mapping class groups to orbifolds, in turn using this to study tilings of covering spaces of orbifolds. Moreover, we study finite subgroups of these mapping class groups. Our results can be used to extend the Delaney-Dress combinatorial encoding of a tiling to yield a finite symbol encoding the complexity of an isotopy class of tilings. The results of this paper provide the basis for a complete and unambiguous enumeration of isotopically distinct tilings of hyperbolic surfaces.
\end{abstract}

\section{Introduction}
The enumerative approaches of Delaney-Dress tiling theory~\cite{DRESS1987} in the two-dimensional hyperbolic plane have facilitated a novel investigation of three-dimensional Euclidean networks, where hyperbolic tilings of triply-periodic minimal surfaces (TPMS) are used for an enumeration of crystallographic nets in $\mathbb{R}^3$~\cite{Sadoc1989,Nesper2001,Robins2005,Robins2006,Ramsden2009,Castle2012}. By relating in-surface symmetries of the TPMS to ambient Euclidean symmetries~\cite{Perez2002,Hyde2014}, the problem of graph enumeration and characterisation in $\mathbb{R}^3$ is transformed to a two-dimensional problem in equivariant tiling theory. The idea is that tilings of the hyperbolic plane can be reticulated over the surface to give a Euclidean geometry to the tile boundaries. This idea has been explored in several contexts over the past $30$ years, including standard hyperbolic tilings by disk-like tiles with kaleidoscopic symmetry~\cite{Ramsden2003,Ramsden2009}, infinite tiles with network-like boundaries~\cite{Hyde2000,hyderams,evansper1,evansper3,Kirkensgaard2014}, and infinite tiles with geodesic boundaries~\cite{evansper2}. 
Chemically, the approach is motivated by the confluence of minimal surface geometry and the structural chemistry of zeolites and metal-organic frameworks~\cite{Hyde1991,Hyde1993,Andersson1984,Chen2001}. In particular, this approach has led to new insights into the structural properties of chemical frameworks in $\mathbb{R}^3$~\cite{Hyde1994}. 

The enumeration of hyperbolic tilings with a given symmetry group reduces down to a problem of enumerating all embeddings of graphs on the orbifold associated to the symmetry group of a tiling, as well as a suitable notion of equivalence among different tilings. 
Delaney-Dress tiling theory provides a systematic approach to the complete enumeration of combinatorial equivalence classes of tilings in simply connected spaces. Computer implementations of algorithms based on Delaney-Dress tiling theory can exhaustively enumerate the combinatorial types of equivariant tilings in simply connected spaces of constant sectional curvature~\cite{Huson1993}. This gives us a description of all combinatorially distinct tilings of an orbifold. For our purposes, we require an understanding of the distinct ways in which this combinatorial structure can be embedded on the orbifold, which in turn represent isotopically distinct tilings of the hyperbolic plane. For example, the Stellate orbifolds 2223 and 2224 can be decorated by a simple combinatorial structure consisting of a single edge. However, this simple structure can manifest as an infinite set of isotopically distinct embedded hyperbolic tilings~\cite{evansper1,evansper2,Pedersen2017,Pedersen2018}.

The classification of embedded combinatorial structures is precisely what this paper will address. We will generalize Delaney-Dress combinatorial tiling theory to classify all isotopically distinct equivariant tilings of any hyperbolic surface of finite genus, possibly nonorientable, with boundary, and punctured. By a hyperbolic surface, we always mean a complete finite-area Riemannian surface with constant sectional curvature $-1$ and totally geodesic boundary. We consider here the $2$-dimensional case, however, the related classifications for higher dimensional hyperbolic orbifolds is also briefly discussed. Our approach is constructive and therefore allows, in theory, a complete enumeration of such classes of tilings. 

Since many of the results we derive here are motivated by the EPINET database (Euclidean patterns in non-Euclidean tilings)~\cite{epinet}, we briefly explain the idea behind the enumerative project. In essence, EPINET enumerates symmetric embeddings of graphs into hyperbolic surfaces. The goal is to enumerate symmetric periodic graphs in $\mathbb{R}^3$ by embedding the underlying hyperbolic surfaces into $\mathbb{R}^3$ in a periodic way and such that the symmetries of the surface and hence the graph embedding correspond to symmetries in $\mathbb{R}^3$. For this, triply-periodic minimal surfaces are used. The graphs considered lift to tilings of the universal covering space of the finite topology surface that embeds in the three-torus to produce the triply-periodic minimal surface, which is the hyperbolic plane $\mathbb{H}^2$. The enumerative process then works in the reverse direction, where finding a finite symbol encoding different tilings of $\mathbb{H}^2$ allows the subsequent enumeration of periodic graphs in $\mathbb{R}^3$. Figure \ref{fig:epinet} gives some examples of the correspondence between the hyperbolic tilings and the tilings of the Gyroid minimal surface in $\mathbb{R}^3$.

\begin{figure}[!tbp]
 \begin{subfigure}[t]{0.3\textwidth}
  \centering
    \includegraphics[width=0.7\textwidth]{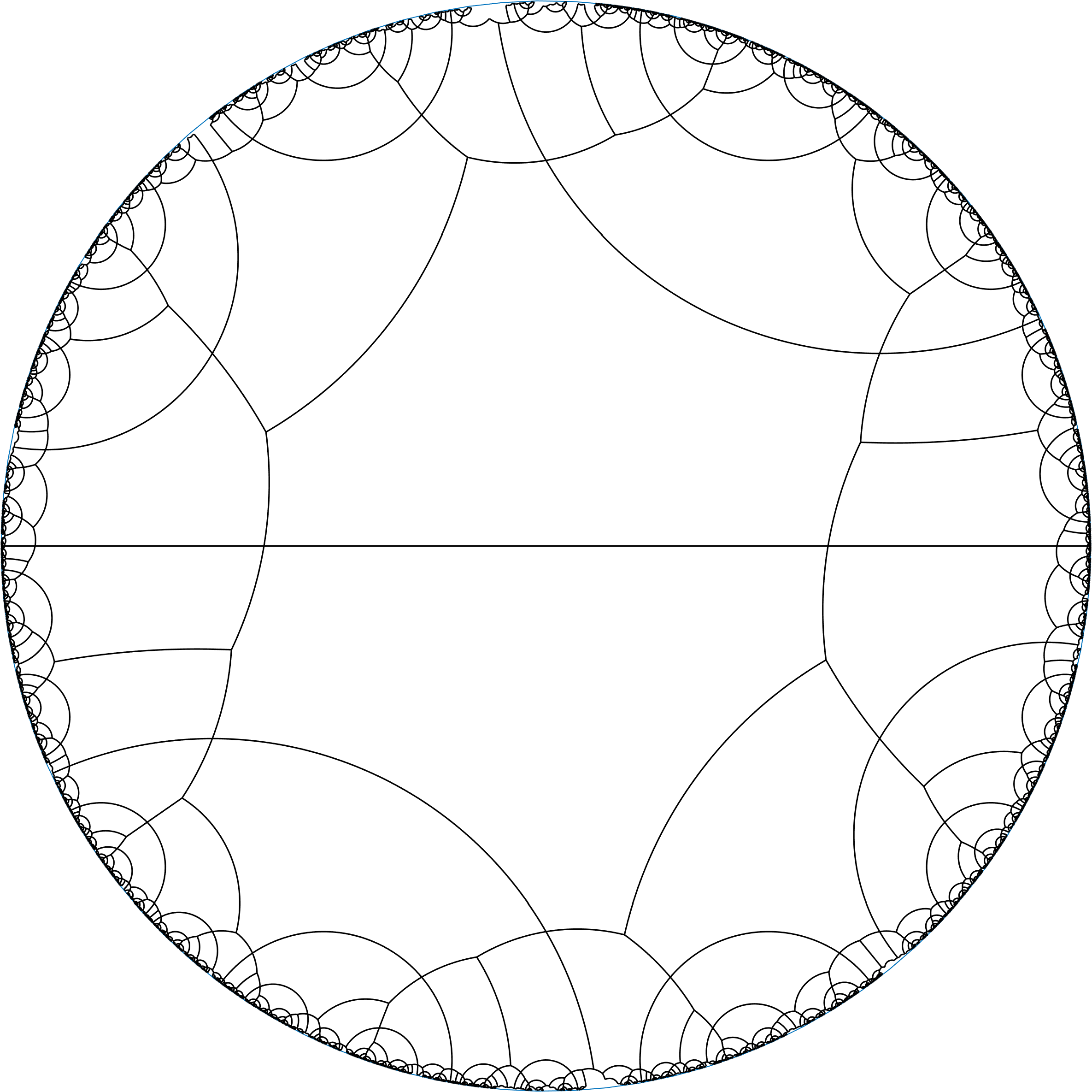}
  \end{subfigure}
  \hfill
  \begin{subfigure}[t]{0.3\textwidth}
  \centering
    \includegraphics[width=0.7\textwidth]{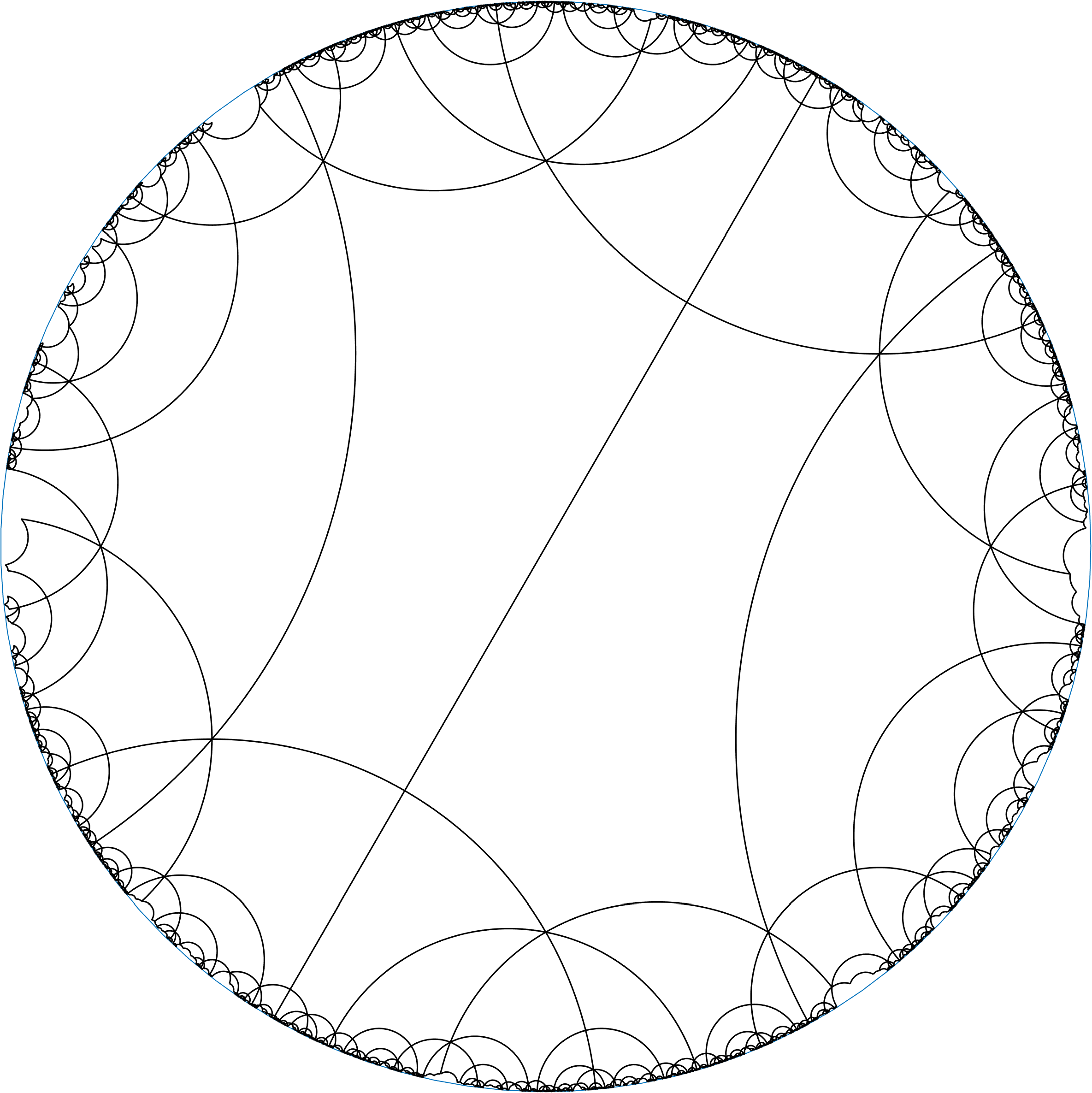}
  \end{subfigure}
  \hfill
    \begin{subfigure}[t]{0.3\textwidth}
  \centering
    \includegraphics[width=0.7\textwidth]{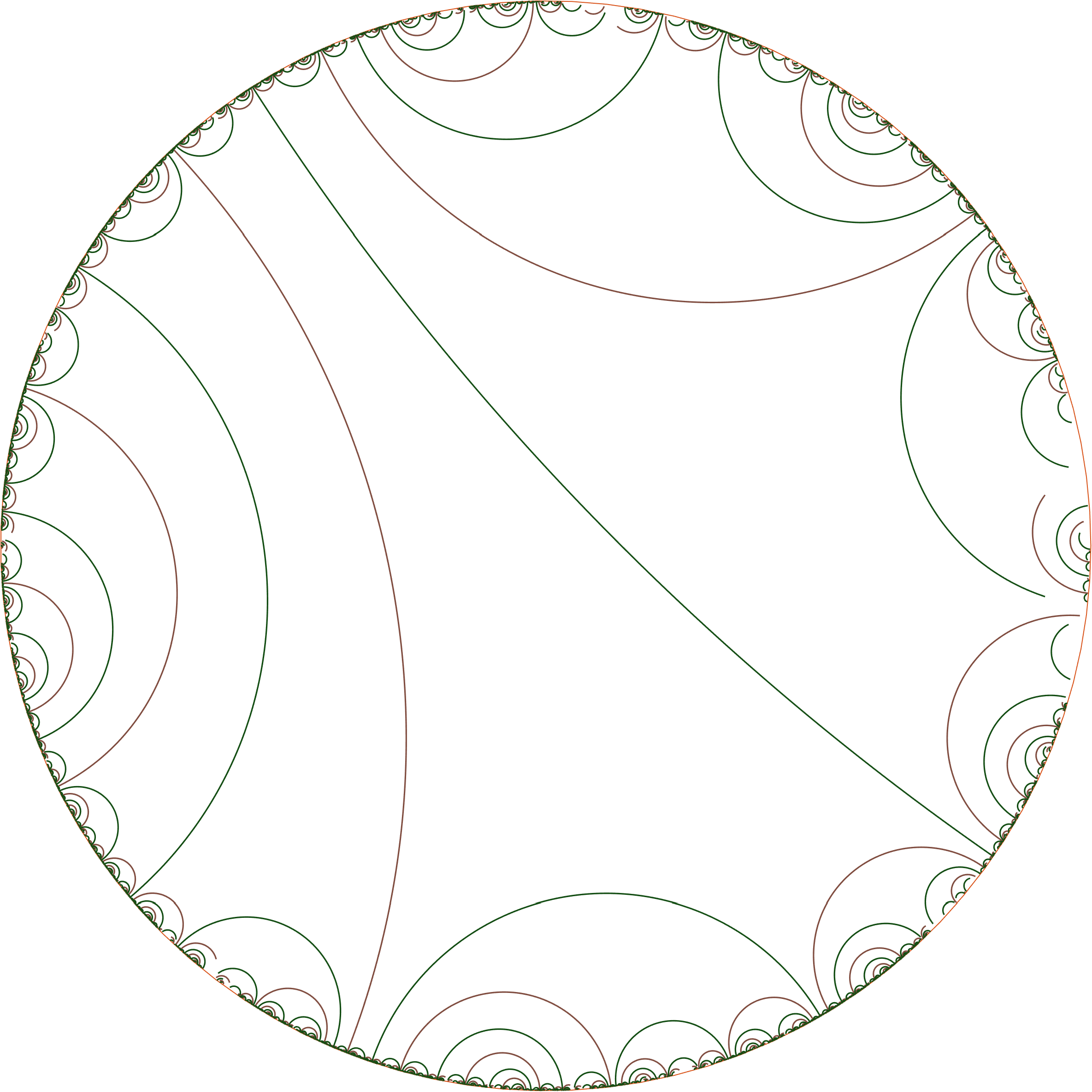}
  \end{subfigure}
  \\
  \begin{subfigure}[t]{0.3\textwidth}
  \centering
    \includegraphics[width=1.1\textwidth]{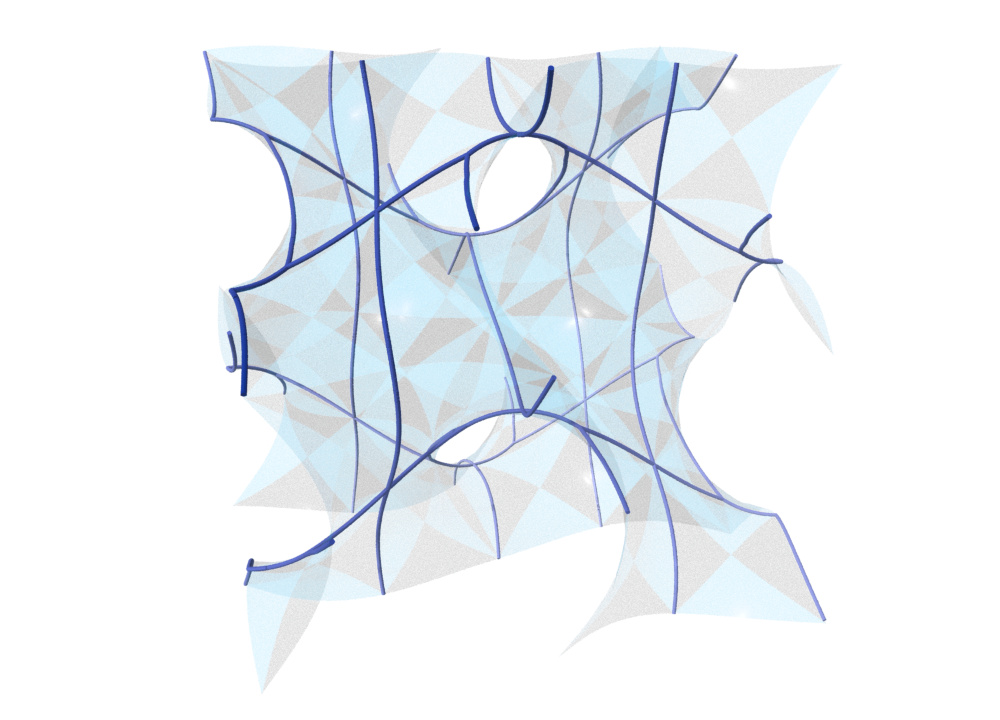}
  \end{subfigure}
  \hfill
  \begin{subfigure}[t]{0.3\textwidth}
  \centering
    \includegraphics[width=1.1\textwidth]{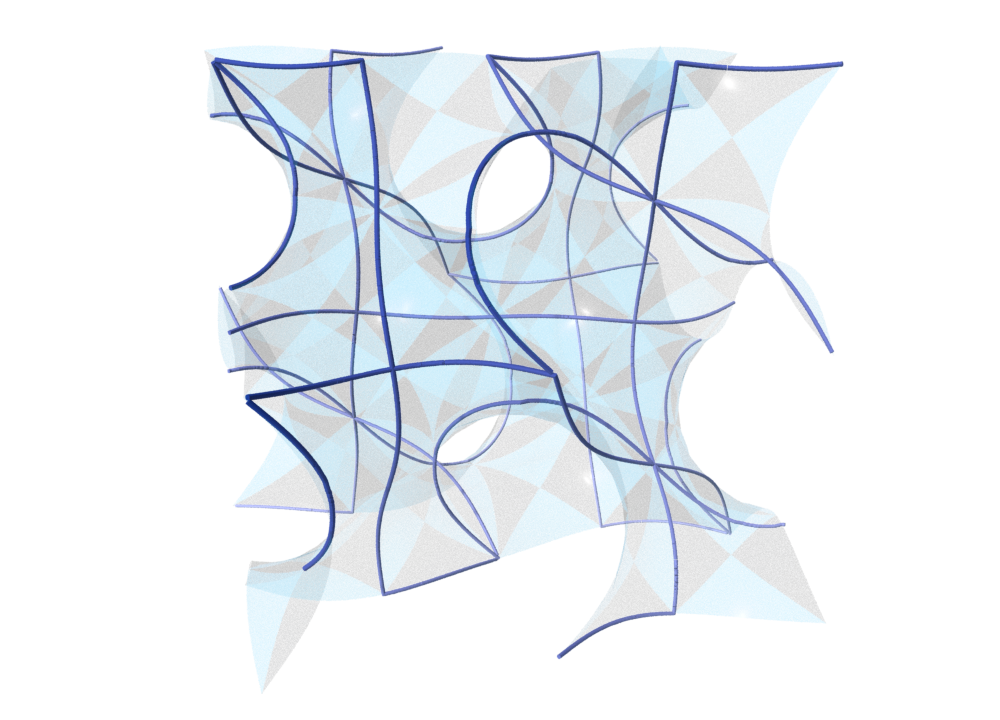}
  \end{subfigure}
  \hfill
    \begin{subfigure}[t]{0.3\textwidth}
  \centering
    \includegraphics[width=1.1\textwidth]{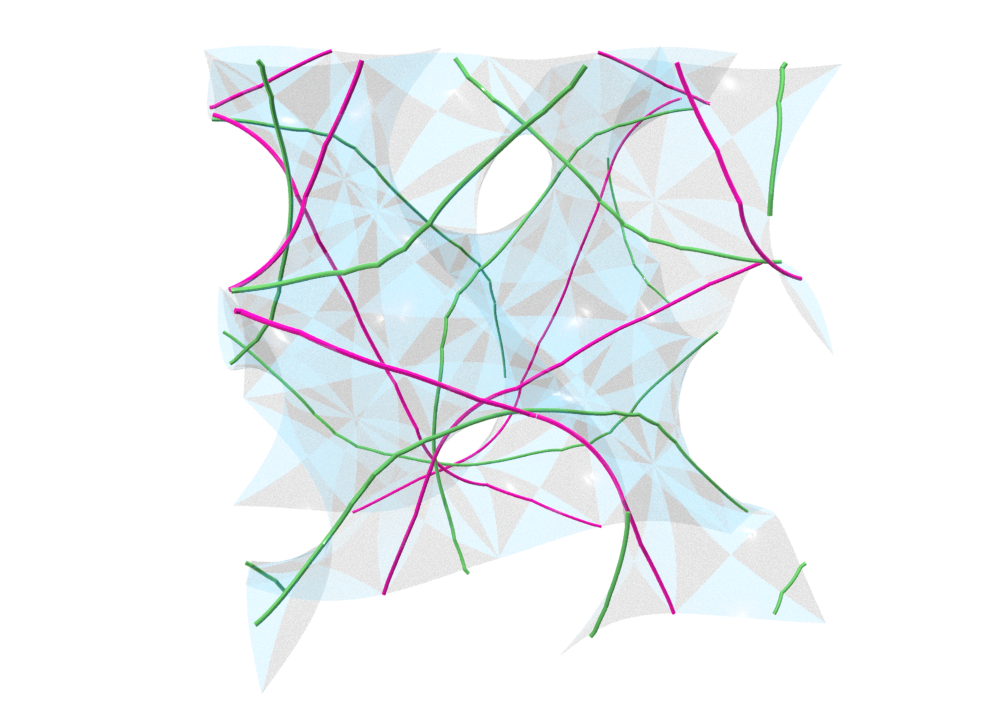}
  \end{subfigure}
\caption{On the top row are shown three tilings of $\mathbb{H}^2$. The black lines show the tile boundaries, each enclosing a tile. These tilings have symmetry 22222, using Conway's orbifold notation. The bottom row shows the edge graphs of the tilings projected onto the Gyroid triply-periodic minimal surface, where one periodic unit cell is shown. The Gyroid is triangulated symmetrically in a way that respects the symmetry group of the surface. Each tiling has the same abstract symmetry group in $\mathbb{H}^2$, and corresponding symmetries in $\mathbb{R}^3$.}\label{fig:epinet}
\end{figure}

Throughout this paper, we make heavy use of the notion of orbifolds~\cite{thurston} and mapping class groups~\cite{primermcgs}. The connection of isotopic tiling theory and mapping class groups is novel, however, there is a well-known connection between the Teichm\" uller space of Riemann surfaces of genus $g$ and certain tilings of the hyperbolic plane with $4g$ sided geodesical polygons that we will use as inspiration~\cite{primermcgs}. We will derive some algorithms to enumerate all equivariant tilings on a hyperbolic Riemann surface in its uniformized metric. Note that this also produces tilings for other Riemannian surfaces by uniformizing the metric within its conformal equivalence class. Indeed, it is well-known that any isometry group of a Riemannian surface gives rise to a unique isometry group of the uniformized surface.

This paper is structured into six sections which cumulatively build the connection between isotopic tiling theory and mapping class groups. We clarify several smaller questions along the way, building intuition of previous results in a new context. We begin with section \ref{sec:orbs} (Symmetry Groups of Tilings and Orbifolds), where we recapitulate the notion of two-dimensional developable orbifolds and expand the framework to incorporate more general classes of orbifolds with punctures and boundary. This is followed by section \ref{sec:ddtheory} (Isotopic Tiling Theory) where we generalize combinatorial Delaney-Dress tiling theory to encode isotopically distinct tilings of surfaces in terms of generators of the symmetry group. In the brief section \ref{sec:outs} (Outer Automorphisms) we will elucidate the connection between outer automorphisms and the generators that encode isotopically distinct tilings. Then, having laid all the groundwork, we will introduce the mapping class group (MCG) of orbifolds in section \ref{sec:MCGorbifold} and prove fundamental results facilitating its applications to tiling theory. In section \ref{sec:bhtheory}, we establish relations between the spaces of tilings of covering spaces and lastly, in section \ref{sec:finiteorders}, we highlight how some of the algebraic properties of MCGs relate to the isotopy classes of tilings they correspond to. In particular, we prove the Nielsen realization theorem in the case of orbifolds and establish the importance of finite subgroups of MCGs for isotopy classes of tilings.
 
This paper represents the theoretical foundation for an enumeration of isotopy classes of tilings on surfaces. The implementation of these results will appear elsewhere and are of inherent interest in the natural sciences. As a result, we make an effort to make the results more accessible by explaining the intuition behind the main ideas.

\section{Symmetry Groups of Tilings and Orbifolds}\label{sec:orbs}
We begin with orbifolds~\cite{Conway2002,thurston,adem2002orbifolds}. Let $\mathcal{X}$ be a simply connected Riemannian manifold $\mathcal{X}$ with constant sectional curvature.  We only work with developable orbifolds, which means that the orbifold $\mathcal{O}$ is the topologically the quotient space $\mathcal{X}/\Gamma$, where $\Gamma\subset \Iso(\mathcal{X})$ is a discrete subgroup. The difference between $\mathcal{X}/\Gamma$ as a topological space and as an orbifold is that for the orbifold structure, one retains the information concerning $\Gamma$ and can reconstruct the topological space $\mathcal{X}$ from $\mathcal{X}/\Gamma$~\cite{ratcliffe2013foundations}. The group $\Gamma$ is called the fundamental group of the orbifold $\mathcal{O}.$ In the classical orbifold setting, $\Gamma$ is required to act cocompactly. We will only require the codomain to have finite area in its uniformized metric, i.e. the metric induced by $\mathcal{X}.$ 

In particular, we are interested in the case $\mathcal{X} = \mathbb{H}^2$, where $\Gamma$ is a NEC group (non-Euclidean crystallographic group), or a hyperbolic orbifold group.

Let $\mathcal{O}$ be a $2$D orbifold. We can identify the symmetry groups using Conway's \emph{orbifold symbol}, as described below, but extended by generators for the non-classical features our orbifolds might have, i.e. hyperbolic transformations $H_i$ of $\mathbb{H}^2$, corresponding to non-mirror boundary components of $\mathcal{O}$ and parabolic transformations $P_j$ corresponding to punctures. The diffeomorphic structure of $\mathcal{O}$~\cite{thurston} is determined by the Conway symbol for its fundamental group $\Gamma:= A\cdots H_i\cdots P_j\cdots \star abc\cdots\times\cdots\circ\cdots$. There are generators for the translations associated to each handle, given by $X$ and $Y$, and going around a handle in an oriented way corresponds to the curve associated to the commutator $\alpha:=[X,Y]=XYX^{-1}Y^{-1}.$ There are also generators for each gyration point of order $A$, and for a curve $\gamma$ going around the gyration point once we have $\gamma^A=1,$ where we interpret the curve as a deck transformation \cite{ratcliffe2013foundations}. For each mirror we have the usual Coxeter group relations, which depend on the angles of the intersecting mirrors. However, in the case where the interior of the orbifold contains nontrivial features, we need to choose one mirror per mirror boundary component that we give two generators $P$ and $Q$, ordered in positive orientation corresponding to its two mirror halves and one generator $\lambda$ for the curve that goes around this boundary component once in positive orientation. In the literature, $\lambda$ is known as the connecting generator for this mirror boundary component. We then add the relation $P=\lambda^{-1} Q \lambda$. Next, going around a crosscap corresponds to a generator $\omega$ with $Z^2=\omega$, where $Z$ corresponds to the curve entering the crosscap once. There is one \emph{global relation} for an orbifold, namely, the product of all Greek letters (plus the nonclassical elements) has to be trivial, i.e. 
\begin{align}\label{eq:globalrel}
\gamma...\Pi_iH_i\Pi_jP_j\lambda...\omega...\alpha...=1.
\end{align}
We shall refer to this presentation as the \emph{standard presentation of the fundamental group} of $\mathcal{O}$. To standardize notation, we can also assume that in the presence of a crosscap, all handles are replaced by two crosscaps each \cite{conwayzip}. In this paper, when we talk about geometric generators of orbifold groups, we generally mean generators of the above form, with a fixed cyclic order as in \eqref{eq:globalrel}.
 
Note that there is a description of the deck transformations in $\Gamma$ as homotopy classes of curves on the orbifold $\mathcal{O}$ \cite{ratcliffe2013foundations,Conway2002}, which we already used above in the description of $\omega$, and which is important to us. We can think of each element in $\Gamma$ as being represented by a homotopy class in the orbifolds (labelled) underlying topological space $O$, introduced in more detail in section \ref{sec:MCGorbifold} below. For this to work, one only has to agree that a closed curve that touches a mirror boundary in $O$ transversally lifts to a curve that crosses over the boundary in $\mathcal{X}$. See also \cite{Conway2002} for an illustration.

The elements of an orbifold fundamental group $\Gamma$ can be assigned types according to their algebraic properties and their action on the hyperbolic plane. Similar to \cite{Maclachlan1975}, we define the type of an element in $\Gamma$ as follows. Torsion elements that preserve the orientation of $\mathbb{H}^2$ are the elliptic transformations of a given order. Mirrors represent torsion elements of order $2$ that reverse the orientation. Orbifold groups like $\Gamma$, when viewed as conformal transformations of the upper half plane $U\subset \mathbb{C}$, have an associated limit set $\Lambda\subset \mathbb{R}$. The complement $C$ of $\Lambda$ in $\mathbb{R}$ has more than one connected component if $\mathcal{O}$ has a boundary. The conjugates of powers of the hyperbolic transformations associated to the boundary components of $\mathcal{O}$ map some component of $C$ to itself and are called \emph{boundary hyperbolic}. There are also the parabolic transformations corresponding to the punctures and the orientation preserving as well as the reversing hyperbolic transformations associated to the genus of a surface.  We call an automorphism (or isomorphism) of $\Gamma$ \emph{type-preserving} if it preserves the types of all elements in $\Gamma.$ We note here that a homeomorphism of hyperbolic orbifolds with symmetry groups $G_1$ and $G_2$, for our purposes, can be defined as a homeomorphism $f$ of $\mathbb{H}^2$ that satisfies $fG_1f^{-1}=G_2$.

\section{Isotopic Tiling Theory}\label{sec:ddtheory}

Tesselations of $\mathbb{H}^2$ can be studied using combinatorial tiling theory \cite{DRESS1987,Huson1993}. Combinatorial tiling theory classifies all possible equivariant combinatorial types of tilings on simply connected metric spaces $\mathcal{X}$ of constant curvature. It deals with the case that each tile is a closed and bounded disk and the symmetry group of the tiling acts cocompactly. 

A locally finite\footnote{This is defined as meaning that any compact set in $\mathcal{X}$ meets only a finite number of tiles.} set $\mathcal{T}$ of such topological disks in $\mathcal{X}$ is called a \emph{tiling} if every point $x\in\mathcal{X}$ belongs to some disk (tile) $T\in \mathcal{T}$ and if for every two tiles $T_1$ and $T_2$ of $\mathcal{T}$, $T_1^0\cap T_2^0=\emptyset,$ where $S^0$ denotes the interior of a set $S.$ 

We call a point that is contained in at least $3$ tiles a \emph{vertex}, and the closures of connected components of the boundary of a tile with the vertices removed \emph{edges}. The only exception to this are two-fold rotational centers of symmetry, which we also  consider to be vertices, depending on the context.
\begin{definition}
Let $\mathcal{T}$ be a tiling of $\mathcal{X}$ and $\Gamma$ be a discrete subgroup of $\Iso(\mathcal{X})$. If $\mathcal{T}=\gamma\mathcal{T}:=\{\gamma T|t\in\mathcal{T}\}$ for all $\gamma\in\Gamma$, then we call the pair $(\mathcal{T},\Gamma)$ an \emph{equivariant tiling} and $\Gamma$ its symmetry group. 
\end{definition}
We call two tiles $T_1,T_2\in \mathcal{T}$ equivalent or symmetry-related if there exists $\gamma\in\Gamma$ s.t. $\gamma T_1=T_2.$ We call the subgroup of $\Gamma$ that leaves invariant a particular tile $T\in\mathcal{T}$ the stabilizer subgroup $\Gamma_T.$ A tile is called \emph{fundamental} if $\Gamma_T$ is trivial and we call the whole tiling fundamental if this is true for all tiles. An equivariant tiling is called tile-, edge-, or vertex-$k$-transitive, if the number of equivalence classes under the action of the symmetry group is $k.$ Note that the above definitions do not require $\Gamma$ to be the maximal symmetry group for the tiling $\mathcal{T}$.
A fundamental tile-1-transitive equivariant tiling (fundamental tiling for short) $(\mathcal{T}, \Gamma)$ has a single type of tile that is a fundamental domain for $\Gamma$ and any fundamental domain for $\Gamma$ also gives rise to such a tiling. The following notion of equivalence among equivariant tilings is central to combinatorial tiling theory.
\begin{definition} 
Two equivariant tilings $(\mathcal{T}_1, \Gamma_1)$ and $(\mathcal{T}_2, \Gamma_2)$ of a simply connected space $\mathcal{X}$ are \emph{equivariantly equivalent} if there is a homeomorphism, $\phi$, of $\mathcal{X}$ and a group isomorphism $h: \Gamma_1 \to \Gamma_2$, such that $\phi(T_1) \in \mathcal{T}_2$ for all $T_1 \in \mathcal{T}_1$ and 
$h(\gamma_1) [(\phi(T_1)] = \phi(\gamma_1[T_1])$ for all $\gamma_1 \in \Gamma_1$. 
\end{definition}
 Intuitively, this definition  means that two equivalent tilings are the same after a change of coordinates that maps the symmetries of one onto the other. We will also sometimes refer to this notion of equivalence as \emph{combinatorially equivalence} among equivariant tilings~\cite{Lucic1990}.

The mechanism behind combinatorial tiling theory that leads to a classification of equivariant equivalence classes of tilings is to consider a barycentric subdivision, invariant under the symmetry group, of the tiles in an equivariant tiling. This subdivision gives rise to a simplicial complex known as a \emph{chamber system}. Keeping track of how chambers gets mapped to one another by elements of the symmetry group of the tiling and subsequent tracking of the combinatorics of the chamber system gives rise to a complete invariant of the combinatorial class of the tiling~\cite{DRESS1987,Huson1993}. 

To generalize the above framework to work with more general symmetry groups and also non-simply connected spaces we detail a slightly different point of view using orbifolds and a concrete realization of a symmetry group $\Gamma$ in $\Iso(\mathcal{X})$, where we particularly emphasize the case $\mathcal{X}=\mathbb{H}^2$. 

One can view tilings as combinatorial structures or classes of decorations on orbifolds. This is based on the simple observation that any tesselation has the symmetry group of a developable orbifold. The underlying topological space of $\mathcal{O}$ can be extracted from any fundamental domain for $\Gamma$ in $\mathbb{H}^2,$ with appropriate edge identifications corresponding to the action of generators of $\Gamma$ on the fundamental domains boundary. The action of $\Gamma$ also gives rise to a fundamental transitive tiling. Each fundamental transitive tile can also be interpreted as a (bordered) fundamental domain and can thus each be seen as a possible canvas, on which we can draw any orbifold decoration (after getting rid of the boundary edges, if necessary). A drawing on a fundamental domain corresponds to a piecewise linear embedding of a graph into the orbifold. In the language of Delaney-Dress tiling theory, each chamber system encoding an equivariant tiling with symmetry group $\Gamma$ essentially corresponds to a triangulation of $\mathcal{O}=\mathbb{H}^2/\Gamma.$ 
 
When viewing tilings as combinatorial decorations on orbifolds, it becomes natural to consider the more general situation of finite volume orbifolds and thus more general symmetry groups for the tilings than for classical Delaney-Dress tiling theory. There are a number of ways of approaching this problem. Section $4$d in \cite{DRESS1987}, includes a sketch of how to adapt the statements made for the case of bounded tiles to work for an equivariant tiling theory for symmetry groups with cusps. To incorporate punctures, one treats the cusps as marked points belonging to the surface and analyses equivariant tilings in terms of chamber systems and geometric cell complexes like in the original setting. As a triangulation of a 2D orbifold, this means that there is a vertex of the chamber system that is placed on a puncture. When embedding these into a manifold to obtain a tesselation, one needs to remove the cusps before embedding. Geometrically, the idea corresponds to pushing the marked points to the boundary of the unit circle in the Poincar\'e model for $\mathbb{H}^2.$  Alternatively, a puncture in the orbifold, corresponding to a parabolic transformation in $\mathbb{H}^2$, can be seen as the limit of a sequence of gyration points of increasing order, with order $\infty$. This is in line with the Conway notation for orbifolds. From this point of view, the tilings for finite volume orbifolds with punctures are attained as limits of tilings for orbifolds where the puncture is a gyration point of increasing order. We will therefore treat the cusped case essentially in the same way as the classical case, but with order $\infty$ singular points.  

The case of a surface with boundary can be treated similarly by simply assuming that the boundary is covered by edges of the tiling for the barycentric subdivision and subsequently treating the chambers along the boundary as neighbouring themselves.

For simply connected spaces and combinatorial classes of tilings, starting from the fundamental tilings, all other equivariant tilings with the same symmetry group are obtained by applying GLUE and SPLIT operations~\cite{delone78,Huson1993}. The different combinatorial types of a fundamental domain for a given classical orbifold were classified in \cite{Lucic1991}. Using the Delaney-Dress ($D$)-symbol, one can give unique names to the combinatorial structures on $2$-orbifolds that represent tilings on $\mathcal{X}$, which can be used for enumeration purposes \cite{Delgado-Friedrichstilings}. We will subsequently focus on fundamental tilings. 
 
Given the generators, described in section \ref{sec:orbs}, of a symmetry group $G\subset \Iso(\mathbb{H}^2),$ the $D$ symbol describes how the group acts on the associated chamber system of a tiling \cite{DRESS1987}, where the chambers are triangles in a triangulation of the orbifold. A fundamental tiling is obtained from a fundamental domain for $G$, with the given generators acting on its boundary edges. By the Poincar\' e theorem, a set of (geometric) generators of $G$ all map part of the fundamental tile's boundary to itself to yield a presentation of the symmetry group, and the $D$ symbol tells us in which way. Note that even if we restrict to geodesically bordered tiles, the $D$ symbol only defines a tiling up to shearing the fundamental domain. 

We now explain a fundamental observation that is one of the starting points of our investigation. The Teichm\" uller space $T(G)$ is the space of type-preserving, discrete faithful representations in $\PGL(2,\mathbb{R})$ (for orientable orbifold groups, one restricts to $\PSL(2,\mathbb{R})$) of the abstract hyperbolic group $G$ with standard presentation, modulo conjugation by elements in $\PGL(2,\mathbb{R})$. This space carries a natural topology, namely the subspace and subsequent quotient topology of $\Hom(G,\PGL(2,\mathbb{R})),$ which itself is endowed with the compact-open topology. The topology of $G$ is the discrete one and $\PGL(2,\mathbb{R})$ carries the topology it inherits from its usual structure as a Lie group. 

As an aside, we comment on perhaps a more classical or standard definition of the Teichm\" uller space of a (topological) surface $S$. Given a hyperbolic surface $X$, one can define a hyperbolic structure on $S$ by using a diffeomorphism $\varphi :S\to X$, known as a \emph{marking}. The Teichm\" uller space is then defined as the space of hyperbolic structures modulo homotopy. Here, two markings $\{\varphi_i:S\to X_i\}_{i=1}^{2}$ are homotopic if there is an isometry $I:X_1\to X_2$ such that $I\circ\varphi_1$ is homotopic to $\varphi_2$. 

We sketch the equivalence of the two definitions of Teichm\" uller space for an orbifold $\mathcal{O}$. Any hyperbolic structure $\varphi : \mathcal{O}\to X$ leads to a Riemannian covering $\pi:\mathbb{H}^2\to \mathcal{O}$, which defines a group of deck transformations up to conjugation by elements in $\Iso(\mathbb{H}^2)$. The marking $\varphi$ induces an isomorphism of $G=\pi_1(\mathcal{O})$ and $\pi_1(X)$ and thus of the deck transformations of the covering by $\mathbb{H}^2$, see section \ref{sec:MCGorbifold} below. Thus, we see that a marking gives rise to a representation of $G.$ For the converse, notice first that a representation $\rho$ of $G$ like in the above definition induces a discrete group in $\Iso(\mathbb{H}^2)$ such that the quotient space has the structure of a hyperbolic orbifold $X$, diffeomorphic to the original $\mathcal{O}$. Now, $\rho$ induces an isomorphism of fundamental groups from $G$ to $\pi_1(X)$. Any isomorphism of fundamental groups of orbifolds is realized as a homeomorphism of $\mathbb{H}^2$(see section \ref{sec:MCGorbifold} below for more details), which induces a homeomorphism between the orbifolds $\mathcal{O}$ and $X$, which is the sought-for marking. The well-definedness is clear because a conjugate of $\rho$ results in an isometric orbifold $X$ and thus homotopic markings. 
		
Most proofs regarding the topological structure of $T(G)$ make use of the second definition. This is essentially because it allows the deconstruction of the surfaces in question into smaller, easier building blocks such as pants or punctured disks. The proofs and ideas in \cite{thurston} that show that $T(G)$ is (component-wise) homeomorphic to $\mathbb{R}^{k}$ for some $k$ also work in our setting with more general orbifolds, when we expand the collection of primitive orbifolds with unique hyperbolic structures that assemble to produce more complicated orbifolds. For this, \cite{orbteich} contains a description of all pieces that are needed for the decomposition if one also allows boundaries. Also see \cite{Kerckhoff1983} for a reference that includes a discussion of the Teichm\" uller space of orbifolds with boundary components for classical surfaces and some of the subtleties involved. For punctures, one can take the pieces that account for rotational symmetries for orbifolds with infinite order. Alternatively, one can view punctures as boundary components with zero length.

The importance of the above is that it implies that two different sets of geometric generators for $G$ in $\Iso(\mathbb{H}^2)$ can be continuously deformed into one another in $\mathbb{H}^2$. The small caveat here is that for orientable $G$, there are two representations with opposite orientation (see section \ref{sec:outs}) in $\PGL(2,\mathbb{R})$, so the connectedness of $T(G)$ is only true for representations of the same orientation. During the process of continuously deforming one representation of $G$ into another, the combinatorial structure of the chamber system associated to the tiling remains invariant. Therefore, perhaps somewhat surprisingly, there is a set of combinatorial instructions for how to decorate the fundamental domain to produce a particular tiling from the generators which is independent of the particular representation of $G$ in $\PGL(2,\mathbb{R})$. This set of instructions can be extracted from the $D$ symbol, see figure \ref{fig:2224} for an example. Within a combinatorial class of fundamental tiles, we can interpret the other fundamental tiles with different positions for the generators as obtained by shearing the original one. This deformation can in fact be realized by a quasi-conformal mapping.

Another, more constructive way of extracting a set of combinatorial instructions for how to decorate the fundamental domain to produce a particular tiling from given generators which is independent of the particular representation of $G$ in $\PGL(2,\mathbb{R})$ is by first constructing the dual tiling. For a fundamental tile transitive tiling with symmetry group $G$, there is a set of so-called Wilkie generators for the symmetry group, corresponding to edge traversals of a fixed copy of any tile~\cite{Wilkie1966}. Expressing these in terms of the given generators, one constructs a tiling with one vertex orbit, dual to the original tiling.

Consider now the example of a fundamental $4g$ polygon of a closed hyperbolic Riemann surface $S$ of genus $g$ with given hyperbolic metric. The construction of a tiling starts from a given point $x\in S$ which is the base point of the generating curves $\{\gamma_i\}_{i=1}^{2g}$ for the fundamental group of $S$. Within each of the homotopy classes for the closed curves $\gamma_i$, there is a unique geodesical representative. Cutting the surface along these geodesics produces a hyperbolic tile and tesselation. While the homotopy classes of the $\gamma_i$ determine the combinatorial structure of the tiling, the choice of base point for the construction of a fundamental domain for the generators produces a plethora of metrically distinct fundamental tilings, which are all derived from the same point in $T(G).$ What different types of fundamental tilings can we create in this way? Any other tiling starts from a different point $p\in S$, and there is a path $c$ connecting $x$ and $p$. The path gives rise to an isotopy of $S$ by pushing the point $x$ along $c$ to $p$. In this way, $c$ uniquely determines a homeomorphism up to isotopies from $(S,x)$ to $(S,p)$, following results relating to the point-push map in \cite{primermcgs}. In particular, if one fixes a reference set of generators of $\pi_1(S)$, the resulting isotopy only leaves the set of generators invariant on $S$ if the path $c$ induces a trivial homeomorphism up to isotopy. This means that if we fix what the generators $\{\gamma_i\}$ map to in $\Iso(\mathbb{H}^2),$ there is only one isotopy class of tilings associated to $S$ and $\{\gamma_i\}$. The combinatorial information needed to produce the corresponding tiling is then simply given by any base point needed to construct the associated fundamental domain. The case $p=x$ is of particular interest. For nontrivial curves $c$, the generators in $\Iso(\mathbb{H}^2)$ change, but because the induced automorphism on $\pi_1(S)$ by any such curve $c$ is always inner, by the Dehn-Nielsen-Baer theorem \cite{primermcgs}, it corresponds to the isotopically trivial homeomorphism of the surface and does not change the tiling in $\mathbb{H}^2$.  

The same line of reasoning works for more general orbifolds $\mathcal{O}$ and their fundamental groups $G$, possibly with repeated use of the above arguments for more than one randomly chosen point, and explains why within a fixed set of generators for $\pi_1(\mathcal{O})\subset \Iso(\mathbb{H}^2)$ and combinatorial type of fundamental tiling with the generators acting on its boundary, the isotopy type of decoration does not depend on the choice of random points required in the construction, up to inner automorphisms of the chosen generators. Note that this does not mean that different sets of generators with the same combinatorial decorations never yield the same isotopy class of fundamental tilings, because the decoration of an orbifold that gives rise to a fundamental tiling is in general not sufficiently complicated to keep track of arbitrary changes. Sufficiently complicated decorations where this cannot happen always exist, as we shall see using more technical arguments below.

Generally, the edges of the fundamental tile can be given purely in terms of the generators, as edges connecting symmetry points, or randomly chosen points. The random points show up in the triangulation that is the chamber system of the orbifold when a vertex is not located at an increased symmetry site. This situation can be read off the $D$ symbol. Using this approach to fundamental tilings from the chamber system related to the $D$-symbols gives a completely algebraic/combinatorial way of producing the fundamental tilings from the generators of $G.$ In practice, this invariant description in terms of generators comes from simply producing a combinatorial version of a tiling from the $D$-symbol and then placing the vertices in the associated decoration accordingly, see figure \ref{fig:2224}. In doing so, the vertices have to be given in terms of their positions relative to the generators.\footnote{This reasoning also works for higher dimensional hyperbolic orbifolds. However, as a result of Mostow's rigidity theorem for developable hyperbolic orbifolds (see \cite{Prasad1973}, \cite{Mostow1968}, \cite{Boileau2004}) there is only one set of generators for the corresponding group and any way to produce all combinatorially distinct fundamental tilings works, without the need to reference specific generators.}

\begin{figure}[!tbp]
\captionsetup{justification=centering}
  \begin{subfigure}[t]{0.31\textwidth}
  \centering
    \includegraphics[width=\textwidth, height=\textwidth]{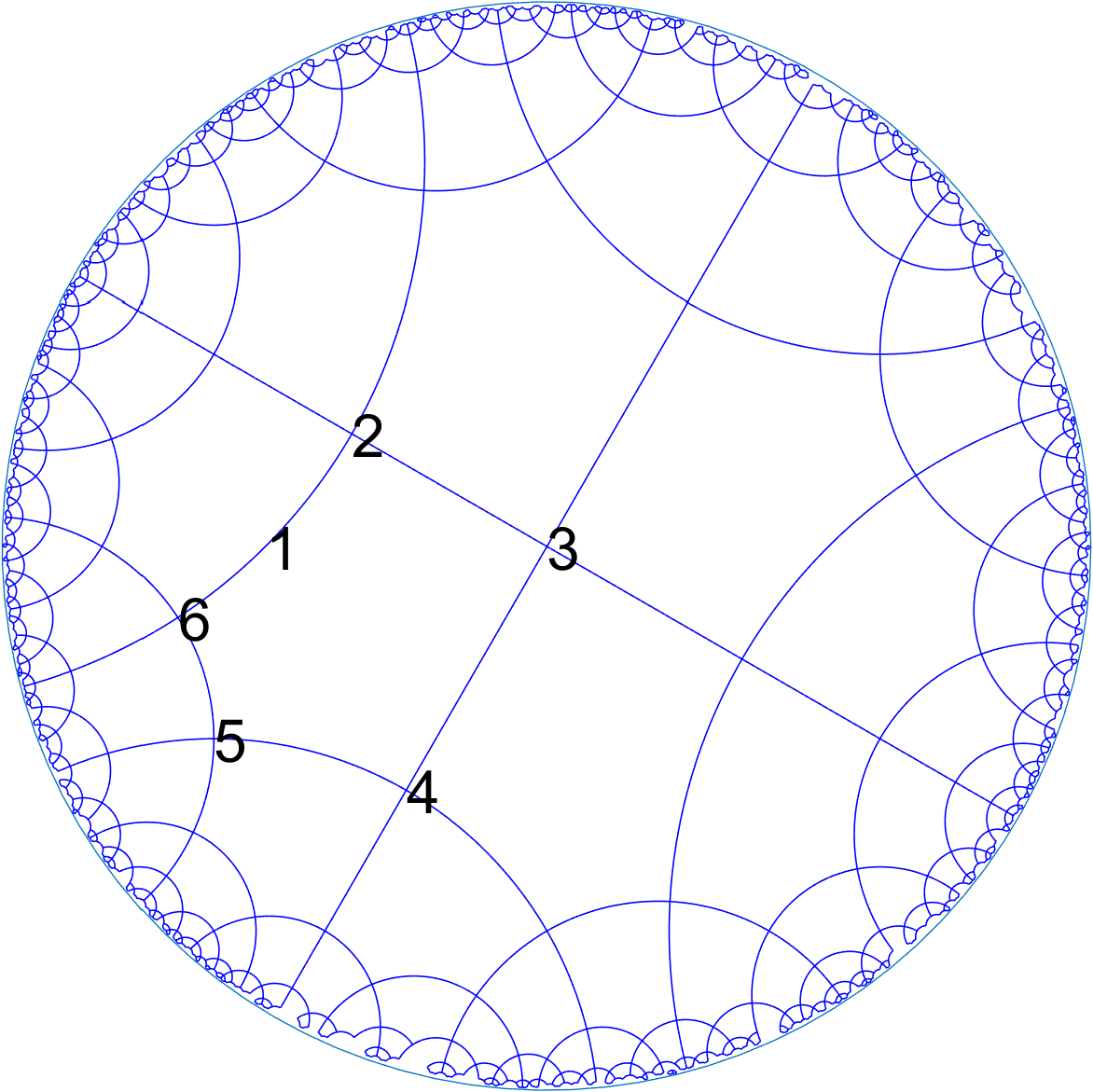}
    \caption{A transitive, fundamental equivariant tiling of $\mathbb{H}^2$ with symmetry group $2224$, generated by rotations around the points labelled $1$ to $4$.}\label{fig:2224ft1}
  \end{subfigure}
  \begin{subfigure}[t]{0.31\textwidth}
  \centering
    \includegraphics[width=\textwidth, height=\textwidth]{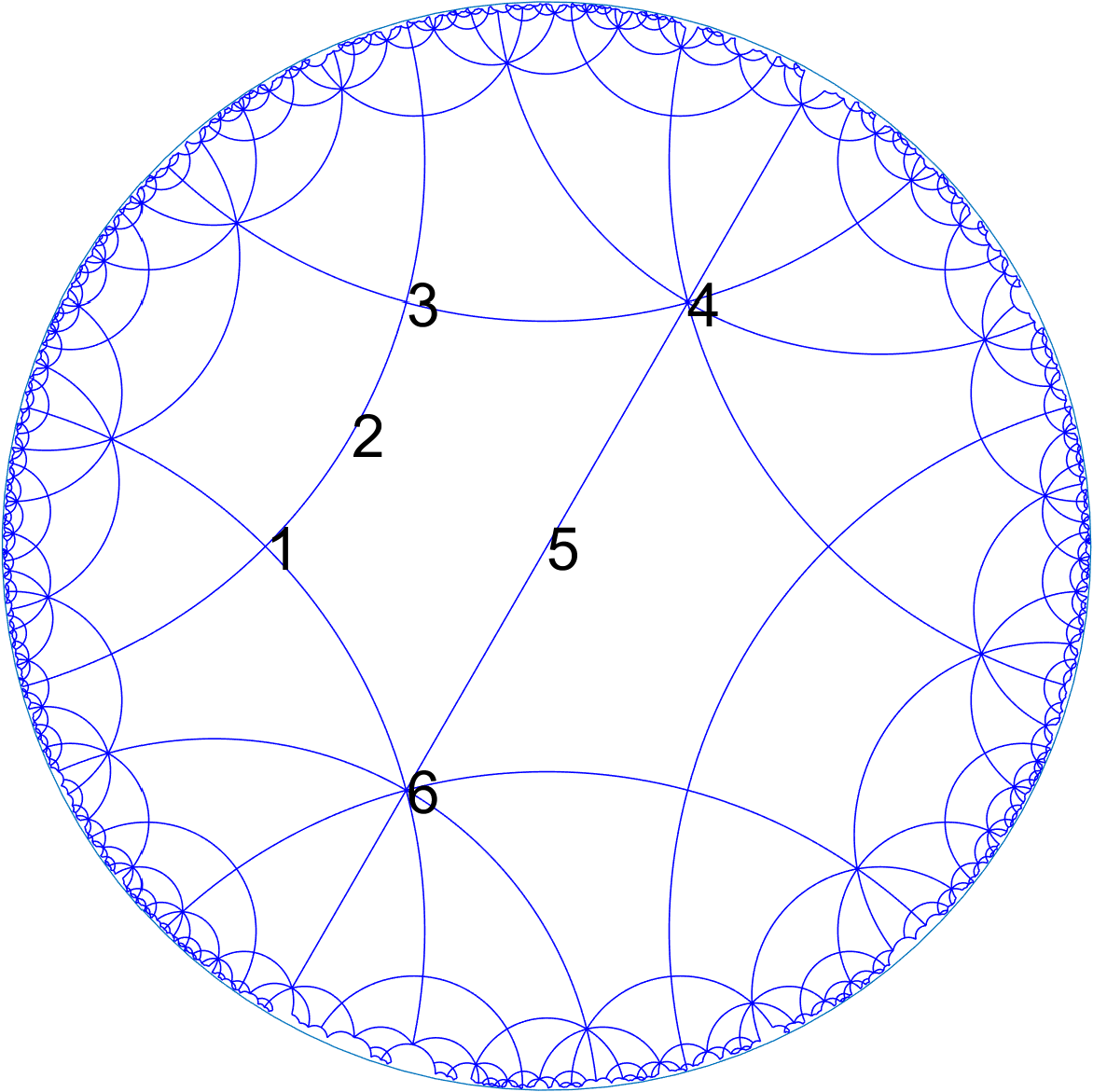}
    \caption{Another transitive, fundamental equivariant tiling of $\mathbb{H}^2$ with symmetry group $2224$, with positions of generators as in (a).}\label{fig:2224ft2}
  \end{subfigure}
  \begin{subfigure}[t]{0.31\textwidth}
  \centering
    \includegraphics[width=\textwidth, height=\textwidth]{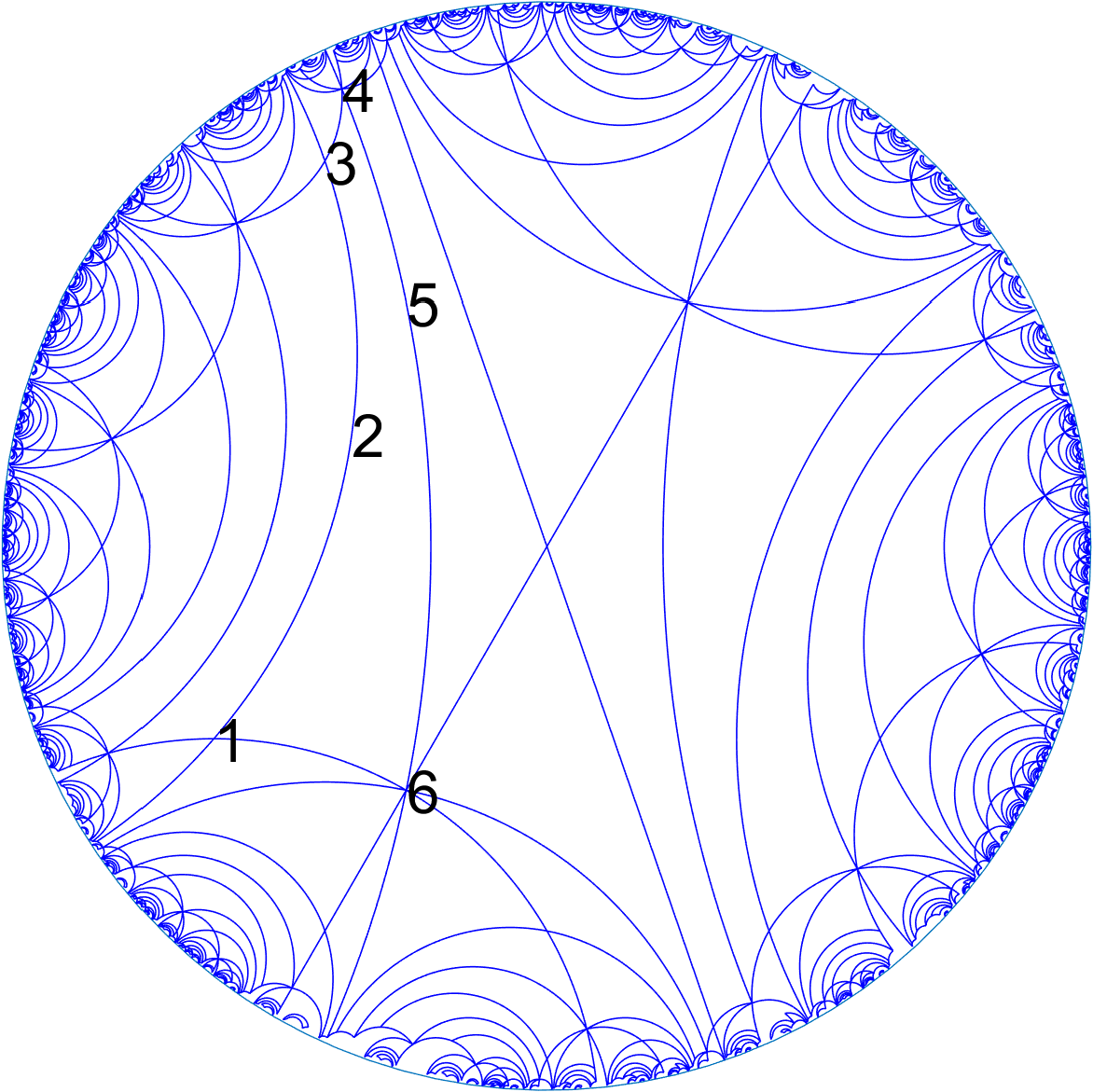}
    \caption{A different placement of generators, producing a sheared version of the fundamental tiling}\label{fig:2224othergens}
  \end{subfigure}
  \caption{Fundamental tilings with symmetry group $2224$, produced by constructng the convex hull of the indicated points, whose position is fixed by the positions of a subset of them that correspond to generators of $2224$. Repeated applications of these symmetires then generates the whole tiling.}\label{fig:2224}
\end{figure}

As an illustration, consider figure \ref{fig:2224}, which shows tilings with a realization of the hyperbolic orbifold group $G=2224$ as a group of isometries in $\mathbb{H}^2$. The placements of the generators in $\mathbb{H}^2$ is indicated in figure \ref{fig:2224ft1}(vertices $1$ to $4$ corresponding to the generators $2224$, respectively) and allows a fundamental tiling for the supergroup $\star 2224$ simply by considering the convex hull in $\mathbb{H}^2$ of the points $1$ to $4$. Now, there are two ways that a fundamental domain for $2224$ can exhibit the symmetries of $\star 2224$. One is obtained by reflecting across the axis through the points $1$ and $4$, as indicated in figure \ref{fig:2224ft1}. The other is obtained by doing the same across the axis through points $2$ and $3$, as has been done in figure \ref{fig:2224ft2}. By the above discussion, we can combinatorially give a description of the edges belonging to the fundamental tiling. In figure \ref{fig:2224ft1}, consider the rotations corresponding to the generators $r_1,...,r_4,$ with centers $c_1,...,c_4\in\mathbb{H}^2$. Because the tiling is obtained by doubling the fundamental tiling of $\star 2224$, it is straightforward to see that the corners/increased symmetry points on the polygon's boundary correspond clockwise, starting at $c_1$, to the points  $c_1,c_2, c_3, c_4, r_4(c_3), r_1(c_2)$. This procedure readily generalizes to arbitrary stellate orbifolds, i.e. those with only rotations for generators. 
Given any generators $r_1,...,r_4$ in $\Iso(\mathbb{H}^2)$ for $2224$, this description of edges defines a fundamental tiling, in this case with geodesic edges, regardless of the generators placement in $\mathbb{H}^2$. Similarly, for the fundamental tiling of figure \ref{fig:2224ft2}, the edges are given by hyperbolic lines connecting the points $1$ to $6$ in cyclic order, which correspond to the points $c_1,c_2, r_2(c_1),r_3^{-1}(c_4),c_3,c_4,$ respectively. Figure \ref{fig:2224othergens} illustrates that this relation for the edges still holds in a sheared version of the fundamental tiling with symmetry group $G$. Here, the sheared fundamental domain does not have the additional symmetries of the other two tilings.

Recall that we want to classify equivariant tilings of a hyperbolic Riemann surface $S$ in its uniformized metric, i.e. given a fundamental hyperbolic polygon of $S$ in $\mathbb{H}^2$, we want to find all ways of equivariantly tiling it, with fixed symmetry group $G\subset \Iso(S)\subset \Iso(\mathbb{H}^2).$ The above suggests an appropriate notion of equivalence for this is to consider equivariant tilings with the same symmetry group that are isotopic in $S$ equivalent. We will see below in section \ref{sec:bhtheory} that two tilings being isotopic in this sense is equivalent to the more strong assumption that two equivariant tilings are equivalent if there is an isotopy between them that preserves their symmetries at every step, which we sum up as follows.
\begin{definition}
Two equivariant tilings with the same symmetry group are isotopically equivalent if they are equivariantly equivalent such that there exists a homeomorphism as in the definition of equivariant equivalence that is isotopic to the identity through a path of homeomorphisms, each of which preserves the symmetry group at every step.
\end{definition}
As far as isotopic tiling theory is concerned, it is not enough to consider just the abstract group $G$ and the associated $D$ symbols in our more general setting. Instead, it is important to use the method of producing fundamental tilings from $D$ symbols along with specific generators for $G$ as outlined above. There is a way to carefully choose only those sets of `locations' for generators for $G$ that yield \emph{a priori} different fundamental tilings of $S$ (see sections \ref{sec:outs} and \ref{sec:bhtheory} below). 

Consider tiling the genus $3$ fundamental polygon of the Riemann surface $S$ in $\mathbb{H}^2$ with symmetry group $\star 246$. There are three different versions of the $22222$ subgroup that are supergroups of $\pi_1(S)$. By Hurwitz' theorem, there is a smallest (area-wise) possible hyperbolic group $G_0$ that is a supergroup of $\pi_1(S)$ and all three versions of $22222$ will be a subgroup of $G_0$ and we see that $\star 246=G_0.$ Each version of $22222$ now has to be treated independently of the others when classifying all isotopy classes of equivariant tilings on $S$. Indeed, the fundamental tilings for every possible set of generators for each of these groups are non-isotopic as tilings on $S$ (see section \ref{sec:bhtheory} below).

Before we go on to introduce new tools for tackling the new challenge of finding appropriate sets of generators for the symmetry groups of tilings, we would like to point out how the GLUE and SPLIT operations work in this new setting. To define GLUE and SPLIT for isotopy classes of tilings, we simply fix one way of implementing them as operations on some representative tiling inductively. One could ask if two different sets of generators $S_1, S_2$ for the same group that produce different fundamental tilings lead to the same tiling of $S$ after a sequence of such operations. If this were the case, then firstly the sequence of operations would be different since they are invertible. However, this would mean that these two different sequences of operations, each applied only to tilings derived from $S_1$ would yield combinatorially equivalent tilings. These are equivalent in the classical Delaney-Dress tilings theory, so no additional ambiguity emerges by us distinguishing between tilings associated to different sets of generators for the symmetry groups. In particular, it is, in a way, very natural to consider the isotopy classes of tilings w.r.t. a set of generators for the symmetry group. Furthermore, this result is very important for enumerative isotopic tiling theory. Note, though, that it is possible that two isotopically distinct tilings that are combinatorially equivalent yield isotopically identical tilings after application of GLUE or SPLIT operations. We summarize these observations. 
\begin{proposition}\label{prop:gluesplitiso}
Let $\Gamma$ be an orbifold group and $L$ an exhaustive and unambiguous list of GLUE and SPLIT operations that yield combinatorially distinct tilings from fundamental tile-$1$-transitive equivariant equivalence classes of tilings with symmetry group isomorphic to $\Gamma$. Then $L$ applied to all isotopy classes of fundamental tile-$1$-transitive tilings produced from all nonconjugate sets of geometric generators for $\Gamma$ yields an exhaustive (but not generally unambiguous) enumeration of isotopy classes of tilings with symmetry group $\Gamma$.
\end{proposition}
Whenever an automorphism of the graph in the orbifold quotient space $O$ that gives rise to a tiling is realized as a homeomorphism of $O$, then we have an ambiguity in the enumeration using GLUE and SPLIT. With the results on the Nielsen realization problem in section \ref{sec:finiteorders} below, we have the result that whenever this is the case, there is a realization of the symmetry group $\Gamma$ of the tiling such that the ambiguities are a result of symmetries in a discrete supergroup of $\Gamma.$ In particular, such a situation can be identified from the D-symbol, as the symmetries of a tiling that exist for some realization manifest as automorphisms of the D-symbol graph. As a consequence, the isotopy theory of tilings described here can arguably be best described as an isotopy theory of coloured tilings, where each edge of a tiling is given a different colour, so that they are distinguishable. When the edges of a tiling are coloured, the enumeration of isotopy classes of tilings using GLUE and SPLIT is unambiguous.

\section{The Group of Outer Automorphisms}\label{sec:outs}

Let $G$ be a hyperbolic orbifold group. The group of all automorphisms of $G$ is denoted by $\Aut(G)$. For example, conjugation by any element $g\in G$ induces an automorphism $c_g(\tilde{g}):=g\tilde{g}g^{-1}$ for $\tilde{g}\in G$ . Such automorphisms are traditionally called inner automorphisms, and the normal subgroup of all of them is denoted by $\Inn(G):=\{c_g|g\in G\}.$ Picturing $G$ as a discrete group of isometries of a Riemannian manifold $M$, roughly speaking, after having chosen a fundamental domain $D$ for $G$, we can reconstruct $M$ by applying elements of $G$ to $D$, and $M$ breaks up into copies of $D$. Under an inner automorphism of $G$, the elements of $G$ that are used to construct $M$ from $D$ translate to instead building $M$ from the same pieces in the same way, starting from another copy of $D$. In view of the previous section, any fixed method of constructing a tiling from generators will reproduce the exact same tiling for any set of conjugate generators. For this reason, we are actually interested in the group of outer automorphisms $\Out(G):=\Aut(G)/\Inn(G).$ 

We fix a set of geometric generators $G_1\subset\Iso(\mathbb{H}^2)$ for $G$ and consider subsets $S$ of the elements of $G\subset\Iso(\mathbb{H}^2)$. We are interested in the following question: When does $S$ constitute a set of geometric generators for $G$ with the same relators? Interpreted within the context of group automorphisms, any such set $S$ gives rise to an automorphism $\alpha$ of $G$, by associating corresponding generators via $\alpha$. Since geometric generators satisfy the same relations, $\alpha$ has a well-defined extension to all of $G$, by expressing any $g\in G$ as some word in the generators and imposing the condition that $\alpha$ is a morphism of groups.

In this way, starting from $G_1$, we see that for any other set of generators $S$, we have a corresponding element of $\Out(G)$, where $G_1$ corresponds to the identity morphism. Note, however, that we are not interested in the full group of automorphisms. Instead, we will restrict our attention to the subgroup of type-preserving automorphisms, as defined in section \ref{sec:orbs}. This restriction is exactly what is needed to ensure that the combinatorics of general tilings are invariant when given as decorations of the associated orbifold (see theorem \ref{thm:mcgout} below). 

Any tesselation with symmetry group $G$ is clearly invariant under an inner automorphism of $G$. The converse is also true - the inner automorphisms of $G$ are the only orientable automorphisms that leave invariant any decoration of fundamental domains for compact orbifold groups $G$. We will prove a version of this statement in proposition \ref{prop:BHgen} below. One way to think about this is to look at the relation between orbifold group elements and curves on the orbifold, which in turn can be interpreted as decorations lifted to the universal cover. Thus, when a sufficiently complicated decoration of the fundamental domain is invariant w.r.t. an (orientation preserving) homeomorphism of the underlying orbifold, the underlying homeomorphism must be isotopically trivial in the orbifold because it fixes all curves and therefore orbifold elements. This means that it corresponds to an inner automorphism of $G$ by theorem \ref{thm:mcgout} below. In case of noncompact orbifolds, this statement is only true for \emph{geometric automorphisms}. 
\begin{definition}\label{def:geomauto}
We call an automorphism $\alpha$ of an orbifold group $\Gamma$ \textit{geometric}, if there exists a homeomorphism $f$ of $\mathbb{H}^2$ that is $\Gamma$ fiber-preserving w.r.t. the universal covering of the orbifold by $\mathbb{H}^2$ (or a totally geodesic subspace thereof in case of boundaries) and induces $\alpha$ via $\alpha(\gamma)=f\gamma f^{-1}$, where $\gamma\in \Gamma\subset\Iso(\mathbb{H}^2)$. Equivalently, an automorphism is geometric if it is induced by a homeomorphism of the orbifold associated to $\Gamma$.	
\end{definition}

Any equivariant tiling corresponds to a decoration of an orbifold, and combinatorial equivalence between tilings means that there is a homemomorphism of $\mathbb{H}^2$ that maps the tilings to each other and induces a homeomorphism of the orbifolds. Such a homeomorphism must map boundary components to boundary components, cone points to cone points of the same order, mirror boundaries to similar mirror boundaries, and punctures to punctures.  Therefore, combinatorially equivalent tilings are never related by nongeometric automorphisms. It turns out that type-preserving automorphisms are closely related to geometric automorphisms as we shall see more precisely in theorem \ref{thm:mcgout}. 

Summarizing, we are not interested in the full group of outer automorphisms, because in the general case of orbifolds with boundaries or punctures, the designation of the type of the generator as a hyperbolic translation or a boundary hyperbolic transformation is important to us. Note also that while orientation is a geometric notion, there is an algebraic analogue~\cite{Zieschang1966} that captures the intuition of the geometric notion, so it makes sense for us to talk about the orientation of automorphisms of an abstractly defined hyperbolic orbifold group, without a specific realization of an orbifold as a group of isometries in $\mathbb{H}^2$. Once expressions for decorations in terms of geometric generators are known, the original surface decorating problem reduces to the study of the group of outer automorphisms of a hyperbolic orbifold symmetry group. 

We are now prepared to formulate a result that highlights the importance of the $2$D setting. The Mostow rigidity theorem implies that the deformation space of finite volume hyperbolic structures on an orbifold of dimension $\ge 3 $ is a singleton. In particular, $\Out(\mathcal{O})$ is trivial and once we have chosen generators for the symmetry group, there is no way to obtain other generating sets via a geometric automorphism. In effect, this means that the combinatorial tiling theory for such non-simply connected hyperbolic manifolds is the same as classical combinatorial tiling theory, which does not take into account different sets of generators and all possible isotopy classes of tilings can be attained by randomly choosing points w.r.t. which one produces the Dirichlet fundamental domain. 
   
\section{The Mapping Class Group of an Orbifold}\label{sec:MCGorbifold}

Our goal is to classify all sets of geometric generators for hyperbolic orbifold groups  $\pi_1(\mathcal{O})\subset\Iso(\mathbb{H}^2)$ that lead to different tilings when decorated in a fixed way, according to $D$-symbols. Moreover, we want to investigate rigorously in what sense this construction of tilings from generators gives rise to a unique tiling. We assume some working knowledge of MCGs of classical surfaces~\cite{primermcgs}. We now introduce the mapping class group (MCG) of orbifolds and prove fundamental results facilitating its applications to tiling theory. 

Let $\mathcal{O}$ be a not necessarily orientable hyperbolic $2$-orbifold, possibly with finitely many punctures and some boundary components. Denote by $O$ its underlying topological surface with weighted marked points at conical singularities of order equal to the assigned weight. Punctures can be treated as part of $O$ by assigning the weight label $\infty$, see section \ref{sec:orbs}. Every mirror in $\mathcal{O}$ represents part of a boundary component in $O$, corresponding to the set of points in its singular locus. Now, similar to conical singularities, we label the mirror boundary components according to their type, which is essentially given by its corresponding substring of the Conway symbol for $\pi_1(\mathcal{O})$. This essentially corresponds to marking the corner points of a miror with a label according to substrings of the Conway symbol. Given such a substring of a symmetry group's Conway symbol that designates a mirror boundary component, it is possible to cyclically permute the numbers in the string and, moreover, to reverse their cyclic ordering in the presence of a crosscap without changing $\pi_1(\mathcal{O})$. Even without a crosscap in the symmetry group, one may reverse the cyclic orders for all mirror boundaries simultaneously~\cite{Conway2002}. Given such an equivalence class of substrings of a Conway symbol with representative $\star abc...d$, place a point marked by an $a$ on the boundary and continue inserting the points $bc...d$ counterclockwise on the boundary, producing a labelled boundary.

We call $O$ together with its labels the \emph{labelled underlying surface} of $\mathcal{O}$. Let $\Sigma$ be the singular locus of $\mathcal{O}$, i.e. the (not necessarily isolated) branch point set of the covering $\mathbb{H}^2\to \mathbb{H}^2/\pi_1(\mathcal{O})$ and denote $O_0:=O-\Sigma.$ Note that $O$ and $O_0$ are orientable if mirrors are the only orientation reversing features of $\mathcal{O}$. We can picture $O_0$ as embedded in a surface with usual boundary components in place of the mirror components.

We consider $\Hom(\mathcal{O}),$ the group of homeomorphisms of $O$ that leave invariant the features of $O$, i.e. do not change the weight assigned to marked points, map punctures to punctures, other boundary components to other boundary components and mirrors to mirrors of the same type, meaning that we require the homeomorphisms to map the marked points on mirror boundaries to marked points on mirror boundaries with the same label, which guarantees that the mirrors in-between them get mapped accordingly. We will assume the homeomorphisms to preserve the orientation of an orientable orbifold and point out the distinction only where there is ambiguity. Endowing $\Hom(\mathcal{O})$ with the compact-open topology turns it into a topological group. Denoting by $\Hom_0(\mathcal{O})$ the connected component of the identity in $\Hom(\mathcal{O}),$  we define the \emph{mapping class group} (MCG) of the orbifold $\mathcal{O}$ as $\Mod(\mathcal{O}):=\Hom(\mathcal{O})/\Hom_0(\mathcal{O}).$  In other words, $\Mod(\mathcal{O})$ is the group of isotopy classes of the homeomorphisms in $\Hom(\mathcal{O})$, where at every step of an allowed isotopy the corresponding map is in $\Hom(\mathcal{O})$. It is possible for an element of $\Hom(\mathcal{O})$ to map a mirror boundary component to itself in a nonisotopically trivial way. For example, when $\mathcal{O}$ has a mirror boundary component labelled by $\star abcd abcd$, then the homeomorphism that twists the boundary half-way around itself is an element of $\Hom(\mathcal{O})-\Hom_0(\mathcal{O})$. For boundary components that are not mirrors, in contrast to standard ways of defining the MCG, for which homeomorphisms are required to fix boundaries pointwise~\cite{primermcgs}, this means that all homeomorphisms and isotopies can twist around the boundary.  For any $[g]\in\Mod(\mathcal{O})$ that does not permute boundary components, there is thus always a representative $f\in \Hom(\mathcal{O})$ that fixes the boundary pointwise if $g$ is orientation preserving~{\cite[Theorem 5.6]{Epstein1966}}. There is thus no difference in how nonmirror boundaries and punctures are treated topologically. There is a well-known way to relate the two ways of defining MCGs for (orientable) surfaces with boundaries~\cite[Proposition $3.19$]{primermcgs}. One of the advantages of defining MCGs that fix boundaries pointwise in other contexts is that MCGs of subsurfaces can be related more easily to subgroups of the MCG of an ambient surface because one can extend homeomorphisms that fix boundaries. 

The following is based mainly on ideas from \cite{Maclachlan1975} and \cite{fujiwara}, whose results we generalize and make precise. The MCG is usually defined using homotopies instead of isotopies. For orientable $\mathcal{O}$, {\cite[Lemma 2]{Maclachlan1975}} shows that $f\in \Hom(\mathcal{O})$ is homotopic in $O_0$ to the identity if and only if it is isotopic in $\Hom(\mathcal{O})$ as defined above. Now, the arguments in the proof of the lemma are based almost exclusively on results by Epstein in \cite{Epstein1966}, whose results are also proved for nonorientable surfaces and not necessarily hyperbolic ones. All ordinary boundary components of $O$ are disjoint from those with mirrors. Therefore, all homeomorphisms that are homotopic to the identity on mirror boundary components are treated in the same way as boundary components but disjointly and the proof remains correct word for word. Thus, we have
\begin{lemma}\label{lem:homiso}
Let $f,g\in \Hom(\mathcal{O})$ be such that they are homotopic on mirror boundary components. Then $[f]=[g]\in \Mod(\mathcal{O})$ if and only if $f$ and $g$ are homotopic in $O_0$. 
\end{lemma}
This lemma illustrates that mirror boundaries require a slightly different treatment from other aspects	of the MCG. 

Let $f$ be a representative of an element of $\Mod(\mathcal{O})$ that maps a single mirror boundary component $m$ to itself nontrivially but is otherwise isotopically trivial, i.e. is supported in a neighborhood of $m$. By definition, $m$ can be interpreted as a polygon with edges corresponding to mirrors. Then, $f$ corresponds to a finite number of Dehn twists around $m$ and possibly a reflection, acting on the constituent mirrors. In $O_0,$ $m$ corresponds to a boundary component, and we can talk about the orientation of this component, given locally in a suitable neighbourhood of $m$. For example, a mirror with label $\star abccba$ admits a nontrivial homeomorphism that corresponds to a reflection that reverses the orientation of the boundary underlying $m$ and therefore the orientation of the curve $\lambda$ that goes around $m$ once. In particular, this homeomorphism of $O$ is generally not supported in a small neighbourhood of $m$, since such an orientation reversing homeomorphism maps the curve going around $m$ to the inverse curve in $\pi_1(\mathcal{O})$, which can only leave the global relation invariant if it also changes other generators for sufficiently complicated orbifolds. The subgroups of orientation preserving elements of any subgroup of the dihedral group are well-known to be cyclic. In particular, there exists an element $c$ of $\Hom(\mathcal{O})$ that is supported in a neighbourhood $U\subset O$ of $m$ that generates all nontrivial orientation preserving homeomorphisms of $U$ that leave $m$ invariant as a set. All homeomorphisms that act trivially on $m$ clearly commute with $c$. 

It is easy to see that $c$ is nontrivial when interpreted as an element in $\Mod(\mathcal{O})$. By considering curves that touch mirror boundary curves and lift to curves that cross the corresponding mirror boundaries in the universal covering space, we also see that all powers of $c$ are generally distinct in $\Mod(\mathcal{O})$. Note though, that there is a smallest\footnote{This is meant as having the smallest absolute value.} integer $k$ such that $c^k$ corresponds to a Dehn twist $T_\lambda$ around $m$. In particular, $T_\lambda$ in $\Mod(\mathcal{O})$ generally depends on $\lambda$ and it can happen that $T_\lambda$ has finite order, if $\lambda$ does.

We refer to any element in the group generated by $c$ associated to a mirror boundary as a \emph{mirror twist}. It is known that MCGs have solvable word problem and are finitely generated~\cite{primermcgs}. The proof, using Alexander's method, that classical MCGs have solvable word problem~\cite{primermcgs} generalizes to the case of orbifold MCGs. 

Assume we are given a generating set containing all mirror twists associated to mirror boundaries and only elements of $\Mod(O_0)$ otherwise. Then we shall refer to any homeomorphism that can be expressed without mirror twists as \emph{not containing mirror twists}. These form a subgroup $F$.

As a consequence of lemma \ref{lem:homiso}, the subgroup of $\Mod(\mathcal{O})$ that corresponds to homeomorphisms that do not include the elements $c$ for the mirror boundary components, but may permute allowed boundaries, is a finite index subgroup of the classical MCG $\Mod(O_0)$ of $O_0.$ Note that $\Mod(O_0)$ can permute all boundary components and punctures. 

We choose base points on every mirror boundary component and a set of curves with one curve for each mirror base point that bounces off the base point before going back to the base point in the orbifold, encircling nothing else. Using these, we can detect whether or not a homeomorphism that permutes boundary components includes twists or not by looking at how these curves get mapped to similar curves, and how the base point changes. Now, if an isotopy class of homeomorphisms $f$ of $O$ not containing mirror twists permutes the boundary components $m_1$ and $m_2$, each with corresponding isotopy class of generators of mirror twists $c_1$ and $c_2$ as above, then we see that $f\circ c_1=c_2\circ f$. This means that the subgroup $T$ of $\Mod(\mathcal{O})$ generated by mirror twists around the mirror boundaries, is normal in $\Mod(\mathcal{O})$. Moreover, we find that the short exact sequence 
\[
1\to T\to \Mod(\mathcal{O})\to \Mod(\mathcal{O})/T\to 1
\]
splits, since clearly $F\cap T = \{e\}$ and $\Mod(\mathcal{O})=FT$. Since the twists around the individual boundaries commute, $T$ is isomorphic to $\mathbb{Z}^s$ for some $s$ and  $\Mod(\mathcal{O})$ is a split extension of a finite index subgroup of the mapping class group of $O_0.$

Different versions of MCGs, including some classes of orbifolds, have received considerable attention in the literature. As far as we know, the above is the first instance of a discussion of elements of the MCG of $\mathcal{O}$ that act nontrivially on mirror boundaries.  

Let $\Gamma=\pi_1(\mathcal{O})$, with canonical projection map $p:\mathbb{H}^2\to\mathcal{O}=\mathbb{H}^2/\Gamma$. Consider a connected component $Z$ of $\mathbb{H}^2- p^{-1}(\Sigma)$, where $\Sigma,$ as above, denotes the singular locus of $\mathcal{O}$. Then $p:Z\to O_0$ is a non-branched and regular cover of connected topological spaces. Furthermore, $\pi_1(O_0)$ has generators $X_i$ corresponding to curves around the isolated points of the singular locus. Choose base points $z_0$ and $x_0$ for $Z$ and $O_0$ such that $p(z_0)=x_0$, so we can talk about concrete subgroups of the fundamental groups involved.

By standard covering space theory~\cite[Proposition $1.39$]{Hatcher}, the group of deck transformations $\tilde{\Gamma}$ of the regular cover $p:Z\to O_0$ and $\pi_1(O_0)$ are related by $\tilde{\Gamma}=\pi_1(O_0)/\pi_1(Z)$. Here, we interpret $\pi_1(Z)$ as a (normal) subgroup of $\pi_1(O_0)$ in the usual way, i.e. under the push forward of $p_\sharp$ as $p_\sharp(\pi_1(Z))$. Clearly, $\pi_1(Z)$ equals the normal closure in $\pi(O_0)$ of the elements $X_i^{o_i}$, where the $o_i$ are the orders of the $X_i$ in $\mathcal{O}$, since these are exactly the relations imposed on the generators of $\pi_1(Z)$ when passing over to $\pi_1(O_0)$ with $p$.  Let $f\in\Hom(\mathcal{O}),$ then, by definition, $f:O_0\to O_0$ preserves the order of branching of $p$. We therefore have that $f_\sharp(X_i^{o_i})\in \pi_1(Z).$  This is exactly the criterion~\cite[prop.\ $1.33$]{Hatcher} for the map $f\circ p$ to lift to a continuous map $f_1:Z\to Z. $ Similarly, we see that $f_1$ is actually a homeomorphism.

We will check that $f_1$ can be uniquely extended to the closure of $Z$ in $\mathbb{H}^2$ and then, if necessary by reflections, to a map $f^*$ on all of $\mathbb{H}^2$ (or a totally geodesic subspace thereof, depending on whether $\mathcal{O}$ has a boundary). Take a sufficiently small neighbourhood $U$ of one of the punctures in $Z$. Then $f_1(U)$ has infinite cyclic fundamental group, meaning it is either a punctured disk or an annulus in $\mathbb{H}^2$. The case of the annulus cannot be true, because $Z\backslash f_1(U)=f_1(Z\backslash U)$ is connected. Note for this that the annulus cannot have one boundary component be equal to a boundary component of $Z$, by construction. Thus, $f_1$ permutes the punctures. Similarly, we see that $f_1$ permutes the boundary components of $Z$ and can thus be extended to the closure of $Z$. Then, to obtain $f^*$, we impose the reflections across the boundaries of $Z$ that correspond to mirrors. That this is well-defined follows from $f_1$ being restricted to preserve the types of mirrors.
The extension is unique and the only ambiguity here stems from lifting $f$ to $f_1$, but two such lifts are related by a deck transformation in $\Gamma.$ We obtain an automorphism $\alpha$ of $\Gamma$ defined by $\alpha(\gamma):=f^*\gamma (f^*)^{-1}$. Note that $\alpha$ is defined only up to conjugation by elements in $\Gamma$. By construction, $f^*$ preserves the designated types of elements of $\Gamma$, and is therefore type-preserving. The above construction of $f^*$ is based on the ideas of~\cite{Maclachlan1975}, pages $499-500.$
     
Below we will need the following theorem, which is proved in \cite{Marden2015} for orientable orbifolds.
\begin{theorem}\label{thm:idisot}
Suppose $f\in\Hom(\mathcal{O})$ and assume that $f$ does not contain any mirror twists. Then there is a lift $f^*$ such that the induced automorphism $\alpha$ is the identity automorphism of $\pi_1(\mathcal{O})$ if and only if $f$ is homotopic in $O_0$ to the identity mapping.
\end{theorem}
We will see below in the proof of theorem \ref{thm:mcgout} that mappings of $O$ that are homotopic in $O_0$ and on mirror boundaries yield the same automorphism of $\Gamma$, which deals with one direction. The proof of the other direction requires careful study of the proof in \cite{Marden2015}. The proof works in exactly the same way as presented there, but we need to exchange one of the key ingredients. The following lemma replaces lemma $1$ in \cite{Marden2015}.
\begin{lemma}\label{lem:closedcurvesiso}
Suppose $S$ is a Riemann surface (possibly obtained, like $O_0$, from a surface with features $O$) and $g$ is a homeomorphism of $S$. Let $\alpha$ be a simple closed curve based at $0$ that is disjoint from the boundary. If there exists an arc $c$ from a point $0 \in S$ to $g(0)$ such that $\alpha$ is homotopic to $cg(\alpha)c^{-1}$ for all such $\alpha$, then $g$ is homotopic in $S$ to the identity.
\end{lemma} 
In \cite{Bers1960}, the statement of lemma \ref{lem:closedcurvesiso} for orientable, closed surfaces is proved using hyperbolic geometry of the surface in $\mathbb{H}^2.$ All of the arguments used there also work for the surfaces with features that we study, as long as we keep in mind the following. First, the construction of the lift of a map $f$ in \cite[p.\ 20]{Bers1960} has to be replaced by the construction of $f^*$ given above. Secondly, the orbifold fundamental group $\pi_1(\mathcal{O})$ is generated by based simple closed curves in $O_0$, similar to the classical case. Note that by isotoping appropriately, we can assume that the basepoint is fixed under homeomorphisms of $O_0.$ Lemma \ref{lem:homiso} then yields the statement of theorem \ref{thm:idisot}.

Note also that using the arguments found in \cite[p.\ 152]{Zieschang1966}, one can prove that one obtains orientation-reversing automorphisms of a group by orientation-reversing lifts of orientation-reversing homeomorphisms.

The MCG of a space is often studied by looking at the action of the homeomorphism classes on isotopy classes of curves. For example, let $\mathcal{O}$ be an orbifold with symmetry group $G:=\pi_1(\mathcal{O})=2222a$, with $a\ge 2$, then $\Mod(\mathcal{O})$ is one of two different types of groups. If $a=2$, $\Mod(\mathcal{O})=\Mod(S_5)$, the usual MCG of the $5$-punctured sphere with punctures $p_1,...,p_5$ corresponding to the fixed points of hyperbolic rotations $r_1,...,r_5$. If $a>2$, $\Mod(\mathcal{O})$ is the subgroup of $\Mod(S_5)$ corresponding to those homeomorphism classes that fix the conical singularity $a$. The set of elements of finite order in $G$ is characterisic, i.e. preserved as a set under automorphisms. If, moreover, an automorphism $\alpha:G\to G$ is type-preserving and orientation preserving (see below for an algebraic definition of orientation preserving), $\alpha(r_i)=tr_jt^{-1}$ for some $t\in G$~\cite{Zieschang1966}. It is impossible that this kind of transformation sends an elliptic transformation to a nontrivial power of itself. Indeed, assume that $r^d=trt^{-1}$ for some $d>1$. Then $r^{d-1}=trt^{-1}r^{-1}=[t,r]$ is elliptic. However, the commutator of an elliptic transformation with any other transformation cannot be elliptic~{\cite[pp.\ 191-193]{Friedrich2008}}. Therefore, mapping an elliptic transformation to a conjugate of a nontrivial power of itself can never yield an automorphism of the whole orbifold group, even if it does yield one of the local group. This generalizes an observation made in \cite{evansper1} and \cite{evansper2}, and discussed in more detail and illustrated with pictures in \cite{myfphd}, whereby the placement of rotational centers for generators in certain domains of $\mathbb{H}^2$ is prohibited. 

As mentioned above, automorphisms of $\pi_1(\mathcal{O}),$ where $\mathcal{O}$ is an orientable orbifold, can be assigned an orientation with the expected property that all orientation preserving automorphisms form a subgroup of index $2$ in all automorphisms~\cite{Zieschang1966}. Moreover, orientation reversing automorphisms are exactly those automorphisms that change the left hand side of the global relation \eqref{eq:globalrel} to a conjugate of its inverse, whereas orientation preserving automorphisms map it to a conjugate of itself. We denote with $\Aut^+(\pi_1(\mathcal{O}))$ the subgroup of orientation and type-preserving automorphisms, which contains all inner automorphisms, and with $\Out^+(\pi_1(\mathcal{O}))$ the corresponding subgroup of outer automorphisms. The well-known Dehn-Nielsen-Baer theorem~\cite[Theorem $8.1$]{primermcgs} can be generalized~\cite{Maclachlan1975} to show that 
\[\Mod(\mathcal{O})\cong \Out^+(\pi_1(\mathcal{O})).\] We will prove the following, in much the same way, by providing an explicit isomorphism. We use the ideas outlined in a different context in a similar proof~\cite{fujiwara} as inspiration. 
\begin{theorem}\label{thm:mcgout}
Let $\mathcal{O}$ be a nonorientable hyperbolic orbifold, with nonorientable underlying topological space. Then the MCG $\Mod(\mathcal{O})$ defined above is isomorphic to $\Out_t(\pi_1(\mathcal{O}))$, the group of type-preserving outer automorphisms. If $\mathcal{O}$ is orientable, possibly containing mirrors, then the orientable MCG $\Mod(\mathcal{O})$ is isomorphic to $\Out^+(\pi_1(\mathcal{O}))$, the group of orientation and type-preserving automorphisms.
\end{theorem}
\begin{proof}
The orientable case without mirrors has already been proven~\cite{Maclachlan1975}, so we focus on the other cases. Define a morphism $\varphi: \Mod(\mathcal{O})\to \Out_t(G)$ by $\varphi(f)(\gamma):=f^*\gamma(f^*)^{-1}$ for $\gamma\in G:=\pi_1(\mathcal{O})$, where $f^*$ is the lift of $f$ defined above. Notice that the ambiguity of $f_1$ in the construction of $f^*$ means that $\varphi$ is only defined up to inner automorphisms. Two isotopic maps in $\Hom(\mathcal{O})$ yield the same image in $\Out_t(G)$, so $\varphi$ is well-defined on isotopy classes. Indeed, assume that $f$ is isotopic to the identity and fixes a point 
$x_0\in O_0$. The idea is that an isotopy of $f$ gives rise to a path of homeomorphisms, which defines a path of automorphisms, as follows. Consider the closed loop $\lambda:[0,1]\to O$ based at $x_0$ that corresponds to the path that $x_0$ takes under the given isotopy $\varphi_t$. Then, by mapping a closed loop $l$ based at $x_0$ to the concatenation of first following $\lambda$ until the point $\varphi_t(x_0)$, then going around the deformed curve $\varphi_t(l)$ and back to $x_0$ along $\lambda$ we see that because $G$ and therefore its automorphism group is discrete, the image in $G$ must be constant, so the induced automorphism is inner.

For mirrors, note that mirror twists are isotopic if and only if they induce the same image in $\Out_t(G)$, because they are defined in terms of homotopy classes of simple closed curves that touch the mirror boundaries. Indeed, firstly, \cite[Sections $1.2.6$ and $1.2.7$]{primermcgs} contains a discussion of how to upgrade homotopies of simple closed curves and arcs to isotopies. These isotopies can be further upgraded to smooth isotopies~\cite{Boldsen2009}, from which \cite[Chapter $8$, theorem $1.3$]{Hirsch1976} yields the result that we can extend such isotopies to isotopies of the whole surface. Note also that mirror twists are supported in neighbourhoods of boundaries, and can be applied appropriately after applying other homeomorphisms first without changing the result.  Note, moreover, that we can apply isotopies to all other homeomorphisms before applying any isotopies to mirror twists, without changing the result. All in all, we have that $\varphi$ is indeed well-defined on isotopy classes of homeomorphisms, concluding, in particular, the proof of the missing direction of theorem \ref{thm:idisot}.

In \cite[Theorem 3]{Macbeath1967}, it is proved that any automorphism of a hyperbolic orbifold group with compact codomain is realized geometrically, i.e. induced by a homeomorphism of $\mathbb{H}^2.$ The proof there can be extended to finite area orbifolds using the uniqueness and existence of an extremal quasi-conformal mapping within an isotopy class of homeomorphisms of the hyperbolic plane as given in \cite[p.\ 59, Theorem 2]{Abikoff1980}. The only difference in the proof then is that instead of reducing to the case of a compact surface by passing over to a finite index subgroup, by the positive resolution of the Fenchel conjecture in \cite{fox52,fconjecture,Chau}, we pass over to the fundamental group of a possibly punctured and bordered orientable surface. This means that instead of every automorphism being realized geometrically as in the compact case, we obtain the statement that only the type preserving ones are realized, as this is the case for surfaces with boundaries and punctures. This last statement, instead of using the original Dehn-Nielsen-Baer theorem for compact surfaces, employs theorem $8.8$ from \cite{primermcgs} instead, which on account of us allowing homeomorphisms that are not the identity on the boundary holds for surfaces with boundary as well, as long as the automorphisms considered are type preserving.

We thus conclude that all type preserving automorphisms of $G$ are realised geometrically and therefore $\varphi$ is surjective. 

For injectivity, assuming that $f$ lifts to a homeomorphism $f^*$ that induces an inner automorphism, there is a lift of $f^*$ that is the identity automorphism on $G.$ Along with the above discussion of mirror twists, theorem \ref{thm:idisot} concludes the proof.
\end{proof}
While theorem \ref{thm:mcgout} is an important result, it is as of yet unclear how to use this isomorphism in general for practical purposes. The same proof holds in the Euclidean case, where the surjectivity of the homomorphism is true for the same basic reasons that it is true for hyperbolic orbifolds.\footnote{The only difference is that the theorems on quasi-conformal maps in the proof of the theorem turn into statements about affine maps.}

From the proof of theorem \ref{thm:mcgout} and the fact that geometric automorphisms are type-preserving, we also obtain the following.
\begin{proposition}\label{prop:comporbs}
For a compact orbifold $\mathcal{O}$ without boundary hyperbolic elements and punctures, every automorphism of $\pi_1(\mathcal{O})$ is realized geometrically, so the MCG $\Mod(\mathcal{O})$ is isomorphic to either the group of all outer automorphisms of $\pi_1(\mathcal{O})$ or just the orientation preserving ones, depending on whether or not $\mathcal{O}$ is orientable.
\end{proposition}

Summarizing what this means for tiling theory, recall that we obtain tilings from a given set of generators of the symmetry group with a fixed method by the previous section \ref{sec:ddtheory}. We now know that two tilings produced in this way for a fixed method with a fixed set of generators yield isotopic tilings. This is because the associated outer automorphism of the symmetry group that fixes the set of generators must be trivial by theorem \ref{thm:mcgout}. Moreover, if two sets of generators yield the same tilings regardless of the fixed method for producing the tilings, they will be conjugate in $\Gamma$. The last statement follows simply because if two sets of generators are distinct and not conjugate, then we can decorate the orbifold's quotient space appropriately with closed loops corresponding to generators and see that some of these must be changed non-trivially when applying a non-trivial element of $\Mod(\mathcal{O})$.

Since combinatorial tiling theory is phrased in terms of equivariant tilings, with a prescribed symmetry group, a natural question is how the isotopy classes of tilings symmetric w.r.t.\ one symmetry group behave in relation to similar tilings that are symmetric w.r.t.\ to another symmetry group. In particular, it is essential for our purposes to consider whether it is possible that two isotopically distinct tilings on $\mathcal{O}$ yield isotopically distinct tilings of a finite covering space. For EPINET, it is important to consider whether or not it is possible that two isotopically distinct decorations of an orbifold yield isotopically equivalent tilings of the TPMS. It turns out that this is impossible. In the next section we will study related questions in some detail, by studying lifts of elements of the MCG to a covering space. 

\section{Lifts of Mapping Class Groups}\label{sec:bhtheory}
In an effort to relate the MCGs of some surfaces to the MCGs of covers of the surface, Birman-Hilden theory was introduced~\cite{Birman1973,Birman1972}. The idea is the following. Given a covering map $p:S \to X$ of surfaces, one may look at \emph{fiber-preserving} homeomorphisms $f:S \to S$ that for all $x\in X$ map the fibers $p^{-1}(x)$ to $p^{-1}(y)$ for some $y\in X$. If this is the case, then $f$ induces a homeomorphism on $X$. Conversely, if a homeomorphism $f$ on $X$ lifts to a homeomorphism $\tilde{f}$ on $S$, $\tilde{f}$ must be fiber-preserving. If for any two fiber-preserving homeomorphisms on $S$ that are homotopic as maps on $S$, there is a homotopy passing only through fiber-preserving homeomorphisms, then we say that $p$ has the \emph{Birman-Hilden property}. The importance of this notion is that the MCGs for surfaces are defined through homotopies and in order to relate the MCGs of both spaces, it is useful to know when only isotopic homeomorphisms of $X$ lift to isotopic homeomorphisms of $S$. 

As such, Birman-Hilden theory concerns itself with maps induced on $X$ by isotopy classes of maps on $S$. This leaves open the question of the existence of a lift of a representative of an isotopy class of maps. We will also investigate the question of existence of lifts of homeomorphisms of orbifolds to their covering spaces. It is known that if $p$ is a finite-sheeted branched regular covering map of orbifolds, then $p$ has the Birman-Hilden property~\cite[Theorem $11.1$]{Zieschang1973}. We will derive this result somewhat differently. The following discussion is cast for hyperbolic orbifolds. For an assessment of the Euclidean case, refer to \cite[$\S 9$]{Zieschang1973}.

Let $p:\mathcal{O}_1\to \mathcal{O}$ be a covering map of orbifolds. As usual, we present the hyperbolic orbifold $\mathcal{O}$, possibly with punctures and non-empty boundary, as the quotient of (a totally geodesic subspace of) $\mathbb{H}^2$ by $\Gamma$, where $\Gamma=\pi_1(\mathcal{O})$ is a discrete subgroup of $\Iso(\mathbb{H}^2)$. We have that $\mathbb{H}^2\to \mathbb{H}^2/\Gamma$ is a regular branched cover, where the branch locus is a (possibly non-discrete) nowhere dense set in $\mathcal{O}$. Similarly, we have $\mathcal{O}_1=\mathbb{H}^2/\Gamma_1$ and we naturally have $\Gamma_1\subset\Gamma,$ with each of these groups acting as a group of deck transformations on the universal cover, see also section \ref{sec:orbs}. We are only interested in finite covers, which translates to $\Gamma_1$ having finite index in $\Gamma$, equal to the degree of $p$. For closed orientable surfaces, it is well-known that any finite index subgroup of the fundamental group is isomorphic to the fundamental group of a covering surface, whereas any infinite index subgroup is free~\cite[Sections $4.2.2$ and $4.3.7$]{Stillwell1993}. 

We will start the subsequent discussion with results whose proofs do not, as far as we know, appear in the literature but can be carried out with well-known methods in the field. 

Homeomorphisms of $\mathbb{H}^2$ satisfying $f\Gamma f^{-1}=\Gamma$ yield geometric automorphisms of $\Gamma$, by definition. This is equivalent to $f$ inducing a homeomorphism of the orbifold $\mathbb{H}^2/\Gamma.$ On the other hand, any homeomorphism $f\in \Hom(\mathcal{O})$ of the orbifold $\mathcal{O}=\mathbb{H}^2/\Gamma$ lifts to a homeomorphism $\tilde{f}$ of the universal covering space $p:\mathbb{H}^2\to\mathbb{H}^2/\Gamma$ that is $\Gamma$ fiber-preserving. To see this, refer to the discussion in the previous section where we explicitly construct such lifts. Recall, furthermore, that theorem \ref{thm:mcgout} implies that geometric automorphisms are exactly those that are type-preserving.
\begin{corollary}\label{cor:geoautsubgr}
A geometric automorphism of an orbifold group $\Gamma$ that induces an automorphism on an orbifold subgroup $S\subset \Gamma$ induces a geometric automorphism on $S$.
\end{corollary}

For the following theorem we will mostly follow the proof of theorem $8.2$ in \cite{Zieschang1973}, but produce a stronger result.  
\begin{theorem}\label{thm:submcgs}
Let $G$ be the symmetry group of a hyperbolic orbifold $\mathcal{O}$ and $G_1\subset G$ a subgroup of finite index. Then a geometric automorphism $\alpha$ of $G_1$ is induced by a $G$ fiber-preserving homeomorphism of $\mathbb{H}^2$ iff $\alpha$ is induced by an automorphism $\hat{\alpha}$ of $G$. 
\end{theorem}
\begin{proof}
If $\alpha$ is induced by a $G$ fiber-preserving homeomorphism $f$, then $f$ induces a homeomorphism on the orbifold $\mathbb{H}^2/G$ as well as, by assumption, on $\mathbb{H}^2/G_1,$ and thus induces an automorphism $\hat{\alpha}$ of $G$ that stabilizes $G_1$ in $G$, which proves one direction.

For the other direction, first consider the situation for the at most index $2$ subgroup $\tilde{N}\subset G$ that contains only orientation preserving elements. We further pass to a finite index normal subgroup $N\subset \tilde{N}$ of $G$ without torsion elements, which we can take to be the fundamental group of a possibly punctured and bordered orientable surface~\cite{fox52,fconjecture,Chau}.

Now let $\alpha$ be induced by a homeomorphism $h$ of $\mathbb{H}^2$ such that $\alpha (n)=h\circ n\circ h^{-1}\; \forall n\in N,$ which w.l.o.g. can be chosen to be the uniquely determined extremal quasi-conformal mapping of $\mathbb{H}^2$ satisfying this relation. 

Now define for arbitrary $g\in G$
\begin{align*}
\varphi=\hat{\alpha}(g)hg^{-1}.
\end{align*}
For $n\in N$ we obtain
\begin{align*}
\varphi(n)&=\hat{\alpha}(g)hg^{-1}(n)=\hat{\alpha}(g)h(g^{-1}ng)g^{-1}=\hat{\alpha}(g)h(g^{-1}ng)g^{-1}\\
&=\hat{\alpha}(g)\alpha(g^{-1}ng)hg^{-1}=\alpha(n)\hat{\alpha}(g)hg^{-1}=\alpha(n)\varphi.
\end{align*}
The fourth equality uses $g^{-1}ng\in N$, since $N$ is normal. Now, $\alpha(n)$ and $n$ act as isometries on $\mathbb{H}^2$, hence leave the dilatation of $\varphi$ invariant, so by the uniqueness of extremal maps we obtain $\varphi=h$ and thus $\hat{\alpha}(g)=h(g)h^{-1}.$ Since $g$ was arbitrary, $h$ preserves $G$-fibers, which proves the theorem.
\end{proof}
The next is a generalization of \cite[Lemma 11]{Maclachlan1975}.
\begin{lemma}\label{lem:uniqueext}
If some automorphism $\alpha$ of the hyperbolic orbifold group $G$ induces an automorphism $\alpha|_H$ of a finite index subgroup $H\subset G$, then there is only one extension of $\alpha|_H$ to $G$, i.e. if $\alpha|_{H}=\id_{H}$ then $\alpha=\id_{G}.$ 
\end{lemma}
\begin{proof}
Let $g\in G$. As a finite index subgroup of $G$, $H$ contains a finite index subgroup $N$ that is normal in $G.$ Then for any $n\in N$ we have $gng^{-1}=\alpha(gng^{-1})=\alpha(g)n\alpha(g)^{-1}, $ i.e. $g^{-1}\alpha(g)$ commutes with every element of $N$. Since $N$ is a hyperbolic orbifold group itself, we have $g=\alpha(g),$ because a nontrivial element in $G$ commutes only with elements of a cyclic subgroup it is part of~\cite{Friedrich2008}. Since $g$ was arbitrary, $\alpha=\id_G$.    
\end{proof}

We now give a short proof of the Birman-Hilden property for general orbifold groups, which is somewhat different than that in \cite{Zieschang1973}.
\begin{proposition}\label{prop:BHgen}
Let $S\subset G$ be a finite index subgroup of the hyperbolic orbifold group $G$.
If a $G$ fiber-preserving homeomorphism $\varphi$ of $\mathbb{H}^2$ is $S$-fiber isotopic
to the identity, then $\varphi$ is $G$-fiber isotopic to the identity.
\end{proposition}
\begin{proof}
By assumption, $\varphi$ induces an automorphism $\alpha$ of $G$, which induces $\id_{S}$ on the subgroup $S$, so by lemma \ref{lem:uniqueext} $\alpha=\id_{G},$ which by theorem \ref{thm:mcgout} implies that $\varphi$ is $G$-fiber isotopic to the identity.
\end{proof}
An important technical consequence of the Birman-Hilden property for isotopic tiling theory is that two sufficiently complicated tilings that arise from decorations w.r.t. a non-conjugate pair of sets of generators for the symmetry group $\Gamma$ with orbifold $\mathcal{O}$ are never isotopic in $S,$ where $S$ is any hyperbolic surface with symmetry group $\Gamma$. Indeed, two sufficiently complicated decorations from a non-conjugate pair of sets of generators for $\Gamma$ are isotopically distinct in $\mathbb{H}^2/\Gamma$, so by theorem \ref{thm:mcgout} the map in $\Mod(\mathcal{O})$ exchanging these decorations must be non-trivial. Now, by proposition \ref{prop:BHgen} above, as tilings of the Riemann surface $S$ that finitely covers $\mathbb{H}^2/\Gamma$, the two tilings, even if they are related by a homeomorphism of $S$, cannot be related by an isotopy in $S$. 

Let $\mathcal{O}_1\to \mathcal{O}_2$ be a finite covering of orbifolds, with fundamental groups $G_1$ and $G_2$ respectively, with $G_1\subset G_2$. Then a geometric automorphism of $G_1$, induced by a map $f$ such that $fG_1f^{-1}=G_1$, is extendible to a geometric automorphism of $G_2$, i.e. a homeomorphism of $\mathcal{O}_2$, iff $f$ satisfies $fG_2f^{-1}=G_2.$ Consider the subgroup $\mathcal{L}\subset \Hom(\mathcal{O}_2)$ of homeomorphisms of $\mathcal{O}_2$ that lift to homeomorphisms on $\mathcal{O}_1$ and set $A:=\Hom(\mathcal{O}_2)/\mathcal{L}.$ Two elements $f\mathcal{L},g\mathcal{L}\in A$ are equal iff $fg^{-1}\in\mathcal{L}.$ This implies that the induced automorphisms $A_f,$ $A_g$ of $G_2$ satisfy $A_f(G_1)=A_g(G_1).$ Said in another way, there are as many equivalence classes in $A$ as there are isomorphic versions of $G_1$ in $G_2$ that get exchanged by automorphisms of $G_2$. Now, $A_f(G_1)$ has the same index in $G_2$ for all $f$. Since $G_2$ is finitely generated, there are only a finite number of subgroups in $G_2$ of a given index, so we obtain the following.
\begin{proposition}\label{prop:finitelifts}
Let $\mathcal{O}_1\to \mathcal{O}_2$ be a finite covering of orbifolds and denote by $\mathcal{L}$ the subgroup of homeomorphisms in $\Hom(\mathcal{O}_2)$ that lift to homeomorphisms of $\Hom(\mathcal{O}_1).$ Then there are only finitely many homeomorphism classes in $\Hom(\mathcal{O}_2)/\mathcal{L}.$ In particular, there are only finitely many topologically distinct symmetric graph embeddings into a surface. 
\end{proposition}
Since there is only a finite number of groups that are possible symmetry groups of a given hyperbolic surface, proposition \ref{prop:finitelifts} also proves that there are only finitely many topologically distinct ways of symmetrically embedding any graph on a hyperbolic surface, which is a well known statement.    

The contents of this section open up possible investigations into more refined questions relating to isotopic tiling theory on a hyperbolic Riemannian surface $S$. For example, by the results of this section, in particular theorem \ref{thm:submcgs} and lemma \ref{lem:uniqueext}, we see that elements of the MCG that are supported in a particular subsurface give rise to automorphisms that leave invariant a subset of the generators. It is well-known that the MCG of any surface has generators that are supported in subsurfaces~\cite{GERVAIS2001}. 
 
The results furthermore add to the duality of the description of the MCG as a group of geometric transformations and as a group of algebraic transformations. In particular, the following important related questions can be examined from algebraic or geometric points of view.
\begin{itemize}
\item Which isotopically distinct tilings with the same symmetry group $G$ are related by a homeomorphism of $S$? 
\item How does an element of the MCG of an orbifold relate to the MCG of a covering orbifold?
\end{itemize}
Note that in most cases these questions do not have a generic answer and depend on the set up, i.e. the conformal structure\footnote{Note that we explicitly allow local coordinate changes to be antiholomorphic.} on $S$ and the tiling. 
\section{Finite subgroups of the MCG}\label{sec:finiteorders}
In this section, we prove results on finite subgroups of the MCG of an orbifold. In particular, we present a proof of the Nielsen realization problem for orbifolds. Theorem \ref{thm:mcgout} implies that the MCG essentially does not depend on the orders of the torsion elements of the orbifold group and as a result, abstract results on MCGs are sometimes useful for applications. Note that in our definition of the MCG, where homeomorphisms are allowed to change the boundary, surfaces with boundary do not necessarily have torsion free MCGs, in contrast to the classical situation~\cite[Corollary $7.3$]{primermcgs}.

An important technical aspect of the EPINET enumerative project is that some isotopically distinct tilings of the embedded hyperbolic surface $S$ in question are related by finite order isometries of the surface that lift to symmetries of $\mathbb{R}^3$. When producing nets in $\mathbb{R}^3$, one often only wants one representative of these. A realization of $\mathcal{O}$ as a quotient space of $\mathbb{H}^2$ with symmetry group $G:=\pi_1(\mathcal{O})\subset \Iso(\mathbb{H}^2)$ induces a metric and conformal structure on $\mathcal{O}$. Before discussing the most general results, we will first explain the set-up and problem.
  
A representative of $\hat{f} \in \Mod(\mathcal{O})$ lifts to a transformation $f$ on $\mathbb{H}^2$.
Suppose $f$ acts as an isometry of $\mathbb{H}^2$. Clearly, $f$ satisfies $f\tilde{g}f^{-1}\in G,$ $\forall \tilde{g}\in G$ so, since $f$ is an isometry, $f\in \mathcal{N}(G),$ the normalizer of $G$ in $\Iso(\mathbb{H}^2)$. Now, suppose that $f\in \Iso(\mathbb{H}^2)$ acts trivially by conjugation on $G$. Then, it would have to fix all of the fixed points on the unit circle at infinity of the hyperbolic translations in $G$. Every hyperbolic orbifold sits inside a surface with at least two independent translations and therefore, $h$ fixes at least $4$ points on the unit circle, hence must be the identity. Therefore, we see that $\mathcal{N}(G)$ injects into $\Aut(G),$ where $G$ itself acts as inner automorphisms of $G$. Now, $fGf^{-1}\subset G$ implies that $f$ preserves $G$ orbits and therefore $f$ induces an (anti-)conformal automorphism of the quotient space $\mathbb{H}^2/G$. 
On the other hand, every conformal automorphism of $\mathbb{H}^2/G$ lifts to an isometry of $\mathbb{H}^2$ by the definition of the conformal structure on $\mathbb{H}^2/G$. Clearly, $G$ itself acts trivially on the space of its orbits. Therefore, we actually get an isomorphism of groups $\Iso(\mathbb{H}^2/G)\equiv\mathcal{N}(G)/G$. In particular, we see that the normalizer $\mathcal{N}(G)$ is discrete, since $\Iso(\mathbb{H}^2/G)$ is. Moreover, $f$ can be interpreted as an element of a hyperbolic supergroup of $G.$ Note, however, that different conformal structures of $\mathbb{H}^2/G$ can give rise to different towers of supergroups of $G$. The condition that $f$ acts as an isometry of $\mathbb{H}^2$ depends solely on the conformal structure induced by the realization of $\mathcal{O}$. 

Assume we are given a finite subgroup $H\subset\Mod(\mathcal{O})$. A natural question is whether or not there exists a finite group $\tilde{H}\subset \Hom(\mathcal{O})$ so that the natural projection $ \Hom(\mathcal{O})\to\Mod(\mathcal{O})$ restricts to an isomorphism $\tilde{H}\to H$. This is known as the \emph{Nielsen realization problem} in the case where $\mathcal{O}$ is a classical surface. It turns out that the proofs of many special cases of the Nielsen realization problem generalize directly to our more general setting. For example, \cite[Theorem $9$]{Maclachlan1975} establishes the positive resolution of the problem for finite solvable subgroups of orientable MCGs without mirrors, based entirely on the proof of the classical theorem. 

The idea of the proof is to find a point in Teichm\" uller space that is fixed by the induced action of the periodic MCG element to show the existence of a metric on $\mathcal{O}$ such that $f$ acts as an isometry, in which case its lift to $\mathbb{H}^2$ does too. To make sense of this idea, one has to introduce a version of Teichm\" uller space to orbifolds in terms of complex structures, see~\cite{thurston,orbteich}. The proof actually shows more, namely that for a finite order element $f$ of the MCG, there is a conformal structure on $\mathcal{O}$, or realization in $\mathbb{H}^2$, such that $f$ acts as a conformal map and therefore as an isometry of $\mathbb{H}^2$ after lifting. 

With the above in mind, we turn to the general Nielsen realization problem for a given finite subgroup of $\Mod(\mathcal{O})$. The positive resolution of the Nielsen realization problem for classical surfaces~\cite{Kerckhoff1983} includes a sketch of the proof for orientable orbifolds and also a sketch for how to generalize the results from orientable surface to nonorientable ones, but it is unclear how to interpret it because the paper does not include a definition of the MCG of a general orbifold. We will present a different approach here that naturally follows from of the results of this paper and gives a complete proof of the Nielsen realization problem for our general hyperbolic orbifold groups. 

Recall first that a characteristic subgroup $C$ of $G$ is a subgroup that is invariant under all automorphisms of $G$. This means that $\varphi(C)\subset C$ $\forall \varphi\in\Aut(G)$ and thus also $\varphi^{-1}(C)\subset C,$ i.e. $C\subset \varphi(C)$ so that any $\varphi\in\Aut(G)$ induces an element of $\Aut(C).$ It is well-known that every finite index subgroup of a finitely presented group contains a finite index subgroup that is characteristic.

\begin{theorem}\label{thm:nielsenreal}
For a given finite subgroup $H\subset\Mod(\mathcal{O})$ of a hyperbolic orbifold $\mathcal{O}$, there is a hyperbolic metric on $\mathcal{O}$ s.t. $H$ acts as isometries.
\end{theorem}
\begin{proof}
There is a finite covering of $\mathcal{O}$ by a classical surface $\tilde{S}$~\cite{fox52,fconjecture,Chau}. We further pass over to a finite index subgroup $G$ of $\pi_1(\tilde{S})$ that is characteristic in $\pi_1(\mathcal{O})$. Then, since $G$ does not contain torsion elements, it corresponds to the fundamental group of a classical, possibly punctured surface $S$. By lemma \ref{lem:uniqueext}, we have an injective morphism $\iota$ from $\Aut(\pi_1(\mathcal{O}))$ to $\Aut(\pi_1(S))$. By theorem \ref{thm:submcgs}, a geometric automorphism $\alpha$ of $\pi_1(S)$ is induced by a $\pi_1(\mathcal{O})$ fiber-preserving homeomorphism of $\mathbb{H}^2$ if and only if $\alpha$ is induced by an automorphism $\hat{\alpha}$ of $\pi_1(\mathcal{O}),$ so, by corollary \ref{cor:geoautsubgr}, the image of a geometric automorphism under $\iota$ is induced by a $\pi_1(\mathcal{O})$ fiber-preserving homeomorphism of $\mathbb{H}^2$. Now, given a path of $\pi_1(\mathcal{O})$ fiber-preserving homeomorphisms of $\mathbb{H}^2$ to the identity, each homeomorphism in the path induces the identity automorphism on $\pi_1(\mathcal{O})$ by theorem \ref{thm:idisot} and the injectivity part of theorem \ref{thm:mcgout}, and thus defines a homeomorphism of $S$.  Therefore, $\iota$ induces a well-defined morphism $\iota:\Mod(\mathcal{O})\to\Mod(S).$ By the Birman-Hilden property in proposition \ref{prop:BHgen}, $\iota$ is again injective and we can interpret $\Mod(\mathcal{O})$ and therefore $H$ as subgroups of $\Mod(S)$. In particular, every element of $H$ corresponds to a $\pi_1(\mathcal{O})$ and $\pi_1(S)$ fiber-preserving homeomorphism. Therefore, the subgroup $\tilde{G}$ of $\Mod(S)$ generated by $H$ and the group of deck transformation of the covering $S\to\mathcal{O}$ is finite. The resolution of the classical Nielsen realization problem~\cite{Kerckhoff1983} yields that $\tilde{G}\subset \Mod(S)$ is realized as a group of isometries of some hyperbolic metric on $S$. By construction, this hyperbolic metric on $S$ induces one on $\mathcal{O}$ that is invariant under $H$.
\end{proof}
We can also express theorem \ref{thm:nielsenreal} as follows. For a finite subgroup $H$ of mapping classes of an orbifold $\mathcal{O}$, there exists a hyperbolic metric $g$ on $\mathcal{O}$ induced by some realization of $\pi_1(\mathcal{O})$ as a group of isometries in $\Iso(\mathbb{H}^2)$ such that there is a subgroup $\tilde{H}\subset \Hom(\mathcal{O})$ so that the projection $\tilde{H} \to H$ induced by $\Hom(\mathcal{O})\to\Mod(\mathcal{O})$ is an isomorphism. Moreover, $\tilde{H}$ acts as isometries on $\mathcal{O}$ with metric $g$. In particular, the elements of $H$ can be interpreted as elements of a discrete supergroup of $\pi_1(\mathcal{O})$ in $\Iso(\mathbb{H}^2).$ 

The above discussion is important for $3$D net enumeration, since the TPMS that account for physically interesting tilings have finite groups of symmetries that are induced by orientable ambient isometries in $\mathbb{R}^3$ and therefore yield the same $3$D structures. In such applications, the finite order elements of the MCG are the candidates for transformations one might be interested in disregarding. 

\begin{proposition}\label{prop:finitestabs}
The subgroup of $\Mod(\mathcal{O})$ that leaves invariant the isotopy class of an equivariant tiling $(\mathcal{T},\pi_1(\mathcal{O}))$ is finite.
\end{proposition} 
\begin{proof}
This is essentially an extension of the well-known Alexander's trick~\cite[Lemma $2.1$]{primermcgs}, by which a homeomorphism of the closed disk that is isotopic to the identity on the boundary is isotopic to the identity. 

Now, an equivariant tiling of $\mathbb{H}^2$ corresponds to a $2$-cell embedding of a graph on $M$, i.e. an embedding of a graph $G$ into $M$ for which $M-G$ consists entirely of components that are homeomorphic to an open disk, where $M$ is any surface covered by $\mathbb{H}^2$ . Assume there was a tiling by disks with edge graph $\tilde{G}$ in $\tilde{M}$ that gives rise to a graph $G$ in $M$ such that $M-G$ had a component $C$ that is not a disk. Let $c$ be a noncontractible loop in $C$. Now, $c$ cannot be contractible in $M$ because $G$ in $M$ is connected. Thus, $c$ is noncontractible in $M$ and corresponds to a deck transformation which is a translation. Denote by $\tilde{c}$ a maximal lift of $c$ in the universal cover $\tilde{M}$. Now, $\tilde{c}$ cuts $\mathbb{H}^2$ into two pieces and is entirely contained in a component of $\mathbb{H}^2-\tilde{G}$, where $\tilde{G}$ is the lifted graph of $G$, which contradicts the assumption on the tiling in $\mathbb{H}^2.$

Therefore, a homeomorphism that leaves invariant $\mathcal{T}$ up to isotopies can be interpreted as a graph isomorphism of the decoration $T$ on $\mathcal{O}$, where every smallest cycle in $T$ bounds a disk in the underlying topological space $O$ of $\mathcal{O}$. Now, if a homeomorphism $f$ leaves invariant all edges of $T$ as sets and does not change their orientations, then by the above Alexander's trick, $f$ is isotopically trivial in every disk bounded by tile edges, and therefore, everywhere on $O$.
Note that by construction the isotopies within every tile leave every edge invariant as a set and we easily see that the isotopies on the boundaries can be chosen inductively for compatibility. 

Therefore, the subgroup of all elements in $\Mod(\mathcal{O})$ that leave $\mathcal{T}$ invariant can be interpreted as a group of graph automorphisms of the graph in $O$ that gives rise to the tiling.
\end{proof}

By theorem \ref{thm:nielsenreal} above, every finite subgroup of a MCG $\Mod(\mathcal{O})$ has an interpretation as a supergroup of $\pi_1(\mathcal{O})$ and therefore leaves some tiling invariant. Conversely, if a tiling $\mathcal{T}$ is invariant under a group $G_\mathcal{T}$ of MCG elements, then $G_\mathcal{T}$ is finite and there is a realization of the equivariant equivalence class of $\mathcal{T}$ that is also invariant under the supergroup containing the elements from $G_\mathcal{T}$.

As an example, consider the finite orders of elements of the orientable MCG $M_n$ of an $n$-times punctured sphere, which were studied in \cite{Gillette1967} and \cite{Murasugi1982}. The result is that $m$ is the order of an element in the MCG if and only if $m$ divides $n$, $n-1$, or $n-2$. There exists an intricate connection between the MCG $M_n$ and the braid group that shows that $M_n$ can be generated by the standard generators of the braid group~\cite{primermcgs}. Using these generators $\{ \sigma_i\}_{i=1}^{n-1}$, any element of finite order in $M_n$ is conjugate to one of the following~\cite[Theorem $4.4$]{Murasugi1982}:
\begin{align}\label{eq:mcgspherefiniteorder}
(\sigma_1\sigma_2\cdots \sigma_{n-1})^k,\quad(\sigma_1\sigma_2\cdots \sigma_{n-1}\sigma_1)^k, \text{ or } (\sigma_1\sigma_2\cdots \sigma_{n-2}\sigma_1)^k
\end{align}
for some integer $k.$ In fact, the expressions in \eqref{eq:mcgspherefiniteorder} yield finite order elements for arbitrary $k$~\cite[Comments following theorem $4.2$]{Murasugi1982}. For $k=1$, these elements have orders $n,$ $n-1,$ and $n-2$, respectively.

Using these observations, one can identify the possible additional symmetry groups of tilings with rotational symmetries as subgroups of the MCG.

\section{Summary and Implications for Applications}\label{sec:results}
We have developed a classification of all isotopically distinct equivariant tilings of a hyperbolic surface $S$ of finite genus, possibly nonorientable, with boundary, and punctured. To analyse the situation for a given hyperbolic surface $S$ we take the following steps. First, we find the smallest (in terms of area) possible symmetry group of $S$, which corresponds to a symmetry group $G_0$ of the hyperbolic fundamental polygon belonging to $S$. This smallest symmetry group $G_0$ exists as a consequence of generalizations of the classical Hurwitz theorem~\cite{oikawa}. There are finitely many possible symmetry groups $G$ for tilings such that $G_0\subset G\subset \pi_1(S)$ because orbifold groups are finitely generated. Given such a $G$, we choose a set of geometric generators. From these generators, we obtain a set of fundamental tilings with symmetry group $G$ as a decoration of the associated orbifold $\mathcal{O}.$ The decoration is specified up to isotopy by a combinatorial description from the Delaney-Dress symbol of the tiling. The mapping class group $\Mod(\mathcal{O})$ of $\mathcal{O}$ naturally acts on the set of sets of geometric generators.
 Thus, starting from the classical Delaney-Dress symbol for the fundamental tiling with the starting set of generators, one obtains all other isotopically distinct fundamental equivariant tilings with symmetry group $G$ by repeated applications of $\Mod(\mathcal{O})$. For each of the resulting fundamental tilings, we independently apply the GLUE and SPLIT operations exactly in the same way as in the classical setting to eventually produce all equivariant tilings with symmetry group $G$. 
 
One caveat here is that in some examples it is possible to find two different sets of generators that are nonconjugate but whose fixed method of producing a tiling results in the same isotopy class of tilings. In such situations, the isotopy class of tiling associated to a decoration of the orbifold is left invariant by an element of the MCG. Such a situation of ambiguity can only occur if the decoration corresponding to the tiling is too sparse to detect the changes the generators undergo. By proposition \ref{prop:finitestabs}, the set of elements of the MCG that leaves invariant an isotopy class of tilings must be finite. Another consequence of the proof of proposition \ref{prop:finitelifts} is that if we colour the edges of a classical tiling to distinguish edges, then the MCG acts with trivial stabilizers.

While $G_0$ is the smallest symmetry group commensurate with $S$, this group depends entirely on the hyperbolic finite area metric on $S$. Without reference to any specific hyperbolic structure, there are many possible chains of subgroups that yield potential symmetry groups of $S$. For example, the group $\star 2226$ appears as the smallest fundamental domain of the $H$ surface in \cite{RobinsHsurface}. However, this group does not appear at all as a symmetry group of the $P$ surface in \cite{Robins2004}. Both surfaces are of genus $3$. Also, $\star 246$ has no hyperbolic supergroups, even though $\star 237$ is smaller. 

An essential ingredient in any assignment of MCG elements to the isotopy classes of tilings comes from the fact that our MCGs have solvable word problem. This allows an unambiguous and complete enumeration of all isotopy classes of tilings with coloured edges on hyperbolic Riemann surfaces by an enumeration of MCG elements. We will leave such an enumeration, including tilings without coloured edges, and an analysis of the situation in figure \ref{fig:epinet}(c), where the tiling in $\mathbb{H}^2$ is not by closed disks, and therefore is not dealt with in classical combinatorial tiling theory, for future endeavours.

\begin{acknowledgements}
We thank Vanessa Robins, Stephen Hyde, and Stuart Ramsden of the Australian National University for detailed discussion and guidance. This research was funded by the Emmy Noether Programme of the Deutsche Forschungsgemeinschaft. B.K was supported by the Deutscher Akademischer Austauschdienst for a research stay at the Australian National University.
\end{acknowledgements}

\bibliographystyle{spmpsci}      
\bibliography{bibliography}   

\end{document}